\documentclass [12pt] {article}
\usepackage{amsthm,amsmath}
\usepackage{epsfig}
\usepackage{amssymb}
\usepackage[bottom]{footmisc}
\newcommand {\RR} {\mathbb R}
\newcommand {\NN} {\mathbb N}
\newcommand {\ZZ} {\mathbb Z}

\newtheorem {theorem} {Theorem}
\newtheorem {lemma} {Lemma}
\newtheorem {prop} {Proposition}
\newtheorem {cor} {Corollary}
\newtheorem{rem}{Remark}

\newcommand {\ttbox} [1] {\mbox {\ttfamily*1}} 

\begin{document}

\title{On the Hardy-Littlewood-P\'olya and Taikov type inequalities for multiple operators in Hilbert spaces}

\author{Vladislav Babenko\footnotemark[1]\footnotetext[1]{Oles Honchar Dnipro National University} , Yuliya Babenko\footnotemark[2]\footnotetext[2]{Kennesaw State University} , Nadiia Kriachko\footnotemark[1] ,\\ Dmytro Skorokhodov\footnotemark[1]}
\date{}

\maketitle

\begin{abstract}
    We present unified approach to obtain sharp mean-squared and multiplicative inequalities of Hardy-Littlewood-P\'olya and Taikov types for multiple closed operators acting on Hilbert space. We apply our results to establish new sharp inequalities for the norms of powers of the Laplace-Beltrami operators on compact Riemannian manifolds and derive the well-known Taikov and Hardy-Littlewood-P\'olya inequalities for functions defined on $d$-dimensional space in the limit case. Other applications include the best approximation of unbounded operators by linear bounded ones and the best approximation of one class by elements of other class. In addition, we establish sharp Solyar-type inequalities for unbounded closed operators with closed range. ~\\
    MSC2010 (primary): 47A63, 41A35, 26D10.\\
    MSC2010 (secondary): 41A65.
\end{abstract}

\section{Introduction}


Inequalities for the norms of derivatives estimate the norm of intermediate function derivative in terms of the norms of the function itself and its higher order derivative(s). Its emergence can be traced back to the theorem on derivatives established independently by A.~Knezer in 1896 and J.~Hadamard in 1897 and rediscovered lately by G.\,H.~Hardy and I.\,E.~Littlewood~\cite{HarLit_12} (see also~\cite[\S1]{BKKP}). First sharp constants in inequalities for the norms of derivatives were obtained by E.~Landau~\cite{Lan_13} in 1913 and independently by J.~Hadamard~\cite{Had_14} in 1914. These results were followed by the works of G.\,H.~Hardy, J.\,E.~Littlewood and G.~P\'olya~\cite{HarLitPol_34}, G.\,E.~Shilov~\cite{Shi_37}. One of the most outstanding and fundamental results in this area was established by A.\,N.~Kolmogorov~\cite{Kol_38,Kol_39}. Thanks to his contribution, such inequalities became widely known as Kolmogorov type inequalities. 

Since then inequalities for the norms of derivatives were studied by many famous mathematicians. Among them were S.\,B.~Stechkin, E.\,M.~Stein, L.\,V.~Taikov, Yu.\,I.~Lubich, I.\,J.~Schoenberg, V.\,V.~Arestov, N.\,P.~Kuptsov, A.\,A. Ligun, A.~Pinkus, V.\,M.~Tihomirov, M.\,K.~Kwong, A.~Zettl, A.\,P.~Buslaev, G.\,G.~Magaril-Illyaev, V.\,F.~Babenko, S.\,A.~Pichugov, V.\,A.~Kofanov,  A.\,Yu.~Shadrin, B.\,D.~Bojanov and many others. 

Large attention and interest of mathematical community to the Kolmogorov type inequalities is not surprising. Such inequalities have deep relations with other areas of Mathematical Analysis such as Approximation Theory, embedding theorems, theory of ill-posed problems, theory of optimal recovery, theory of optimal algorithms, and find applications in ODE and PDE. Furthermore, the methods developed for their proof find numerous applications in the study of other extremal problems of Mathematical Analysis. 

Also, inequalities for the norms of derivatives have certain analogues in Functional Analysis. In particular, L.~H\"ormander obtained the following result~\cite[Theorem~1.1]{Hor_55} (see also~\cite[\S 2.6]{Yos_95}).
\begin{theorem}
	Let $X,Y,Z$ be Banach spaces, $A:X\to Y$ be closable operator with domain $\mathcal{D}(A)$ (see definition in~\cite[\S2.6]{Yos_95}), $B:X\to Z$ be closed operator with domain $\mathcal{D}(B)$ (see definition in~\cite[\S2.6]{Yos_95}), and $\mathcal{D}(B) \subset\mathcal{D}(A)$. Then there exists $C > 0$ such that 
	\[
		\|Ax\|_Y \leqslant C\left(\|x\|_X + \|Bx\|_Z\right),\qquad x\in\mathcal{D}(A).
	\]
\end{theorem}

Nowadays results of E.~Landau and J.~Hadamard are generalized in various directions: for uniform and integral norms, for derivatives of integral and fractional orders, for derivatives of multivariate functions, for powers of the Laplace-Beltrami operators defined on manifolds, for powers of infinitesimal generators of semigroups, for powers of self-adjoint operators and abstract linear operators acting in Hilbert and Banach spaces. For a thorough overview of the history of the topic and known results, discussion of related problems and further references, we refer the interested reader to the books~\cite{Mit_Pec_Fin_91, Kwo_Zet_92, BKKP} and surveys~\cite{AreGab_95, Are_96, BabKofPic_97, BabKofPic_98}, and also to papers~\cite{Ily_05, IlyTit_06} for some interesting applications in PDEs. 

The most fruitful setting where a lot of sharp inequalities are established is concerned with operators acting in Hilbert spaces. Under this setting, either the Hilbertian norm of intermediate function derivative is estimated in terms of Hilbertian norms of the function itself and its higher order derivatives, or the value of some given linear functional on the intermediate function derivative is estimated in terms of Hilbertian norms of the function itself and its higher order derivatives. In what follows we will refer to the former estimate as the Hardy-Littlewood-P\'olya inequality and to the later estimate as the Taikov type inequality. We list only some of the advances in this setting:
\begin{enumerate}
	\item Results of G.\,H.~Hardy, J.\,E.~Littlewood, G.~P\'olya~\cite{HarLitPol_34} and L.\,V.~Taikov~\cite{Tai_68} concerning derivatives of functions defined on $\mathbb{R}$. Case of fractional order derivatives was considered by V.\,M.~Tihomirov and A.\,P.~Buslaev~\cite{BusTih_79};
	\item Results of Ju.\,I.~Ljubi\^c~\cite{Lub_60}, M.\,P.~Kuptsov~\cite{Kup_75}, V.\,N.~Gabushin~\cite{Gab_69}, G.\,A.~Kalyabin~\cite{Kal_04}, A.\,A.~Lunev and L.\,L.~Oridoroga~\cite{Lun_Ori_09} concerning derivatives of functions defined on the non-negative half-line $[0,+\infty)$. Case of fractional order derivatives was considered by V.\,M.~Tihomirov and G.\,G.~Magaril-Il'yaev~\cite{MagTih_81};
	\item Results of G.\,H.~Hardy, J.\,E.~Littlewood and G.~P\'olya~\cite{HarLitPol_34} and\break A.\,Yu.~Shadrin~\cite{Sha_90} concerning derivatives of periodic functions;
	\item Results of V.\,M.~Tihomirov and A.\,P.~Buslaev~\cite{BusTih_79}, A.\,A.~Ilyin~\cite{Ily_98}, A.\,A.~Lunev~\cite{Lun_09} concerning derivatives of functions defined on $\mathbb{R}^d$ or the octants in $\mathbb{R}^d$;
	\item Results of A.\,A.~Ilyin~\cite{Ily_98} concerning the powers of the Laplace-Beltrami operators on the sphere $\mathbb{S}^d$;
	\item Results of V.\,F.~Babenko and R.\,O.~Bilichenko~\cite{BabBil_10,BabBil_11} concerning the powers of self-adjoint and normal operators in Hilbert spaces;
	\item Results of V.\,F.~Babenko and N.\,A.~Kriachko~\cite{BabKri_14} and V.\,F.~Babenko, Yu.\,V.~Babenko and N.\,A.~Kriachko~\cite{BabBabKri_16} concerning functions of self-adjoint operators in Hilbert spaces.
\end{enumerate}

In this paper we demonstrate a simple unified yet powerful approach to obtain sharp mean-squared and multiplicative Hardy-Littlewood-P\'olya and Taikov type inequalities. We apply it to prove some of well-known inequalities listed above and also establish a series of new inequalities, in particular, for the norms of the powers of the Laplace-Beltrami operator on Riemannian manifolds and CROSS spaces. 

The paper is organized as follows. We introduce necessary notations and definitions in the next section. In Section~\ref{Sec:Taikov} we obtain sharp mean-squared Hardy-Littlewood-P\'olya and Taikov type inequalities for operators acting in Hilbert spaces. Simultaneously we obtain sharp additive Taikov type inequality and solve related Stechkin problem of the best approximation of unbounded functionals by linear bounded ones on sets defined with the help of self-adjoint operators. We demonstrate how to establish sharp multiplicative inequalities as consequences from corresponding mean-squared inequalities in Section~\ref{Sec:Multiplicative}. In Section~\ref{Sec:Applications} we provide applications of main results to finding the best approximation of a class by elements from another class and to establishing sharp inequalities between the norms of powers of the Laplace-Beltrami operators on Riemannian manifolds and fractional differentiation operators on $\mathbb{R}^d$. Finally, in Section~\ref{Sec:Solyar} we establish sharp Solyar-type inequalities for closed operators having closed range, and obtain sharp Solyar-type inequality for the powers of the Laplace-Beltrami operators on Riemannian manifolds as consequence. 

\section{Preliminaries}

Let $\mathbb{R}_+ = (0,+\infty)$, $H$ be a separable Hilbert space over $\mathbb{C}$ endowed with scalar product $(\cdot,\cdot)_H$, norm $\|\cdot\|_H$ and orthonormal basis $\left\{e_{n}\right\}_{n\in M}$, where $M$ is a finite or countable set. For definiteness but without loss of generality, we fix the meaning of the sum of series indexed by the set $M$. Let $\left\{M_N\right\}_{N=1}^\infty$ be any given sequence of nested finite subsets of $M$ such that $M = \bigcup\limits_{N=1}^\infty M_N$. Then, for a sequence $\left\{a_{n}\right\}_{n\in M}\subset H$ or $\left\{a_{n}\right\}_{n\in M}\subset \RR$, we set: 
\[
	\sum_{n\in M} a_n := \lim\limits_{N\to\infty} \sum\limits_{n\in M_N} a_n
\]
providing the limit in the right hand part exists. In what follows the choice of the sequence $\left\{M_N\right\}$ is not essential as only absolutely convergent series will be considered. For $x\in H$, by $x_n := (x,e_n)_H$, $n\in M$, we denote its Fourier coefficients with respect to the system $\{e_n\}_{n\in M}$ and by $S_Nx := \sum\limits_{n\in M_N} x_ne_n$, $N\in\mathbb{N}$, -- partial sums of the Fourier series of $x$.

For convenience, for sequences $\left\{a_{n}\right\}_{{n}\in M},\left\{b_{n}\right\}_{{n}\in M}\subset[0,+\infty)$, we denote
\[
	\widetilde{\sum\limits_{n\in M}}\frac{a_{n}}{b_{n}} := \left\{\begin{array}{ll} +\infty, & \textrm{if } \exists n_0\in M \textrm{ s.t. } (a_{n_0}\ne 0) \wedge (b_{n_0} = 0),\\[7pt] \displaystyle\sum\limits_{n\in M\,:\,b_n\ne 0} \frac{a_n}{b_n}, & \textrm{otherwise},\end{array}  \right.
\]
and
\[
	\widetilde{\sup\limits_{n\in M}} \frac{a_n}{b_n} := \left\{\begin{array}{ll}+\infty, & \textrm{if } \exists n_0\in M \textrm{ s.t. } (a_{n_0}\ne 0) \wedge (b_{n_0} = 0),\\[7pt] \displaystyle\sup\limits_{n\in M\,:\,b_n\ne 0} \frac{a_n}{b_n}, & \textrm{otherwise}.\end{array}\right.
\]
When set $\{n\in M\,:\,b_n\ne0\}$ is empty we define $\widetilde{\sum\limits_{n\in M}}\frac{a_n}{b_n} := 0$ and $\widetilde{\sup\limits_{n\in M}}\frac{a_n}{b_n} := 0$.


Next, we let $H'$ be a Hilbert space over $\mathbb{C}$ endowed with scalar product $\left(\cdot,\cdot\right)_{H'}$ and norm $\|\cdot\|_{H'}$. For $m\in\ZZ_+$, consider linear operators $B_j:H\to H'$, $j=0,\ldots,m$, with domains $\mathcal{D}\left(B_j\right)$. We will require operators $B_0,\ldots,B_m$ to satisfy some of conditions below:
\begin{enumerate}
	\item [(B1)] $\forall n\in M\Rightarrow$ $e_{n} \in \bigcap\limits_{j=0}^m \mathcal{D}\left(B_j\right) =: \mathcal{D}_{\bf B}$;
	\item [(B2)] $\forall n',n''\in M$, $n'\ne n''$, and $\forall j\in\{0,1,\ldots,m\}$ $\Rightarrow$ $\left(B_je_{{n}'}, B_j e_{{n}''}\right)_{H'} = 0$;
	\item [(B3)] $\exists j_0\in\{0,\ldots,m\}$ and $\exists n_0\in M$ $\Rightarrow$ $B_{j_0}e_{n_0}\ne 0$;
	\item [(B4)] $B_0,B_1,\ldots,B_m$ are closed operators~(see, {\it e.g.}~\cite[\S 2.6]{Yos_95}).
\end{enumerate}

For $j=0,\ldots,m$, we consider the subspace
\[
	H_j := \left\{x\in H\,:\,\|x\|^2_{H_j} := \sum\limits_{{n}\in M} \left|x_{n}\right|^2\left\|B_je_{n}\right\|_{H'}^2 < \infty\right\},
\]
set $H_{\bf B} := \bigcap\limits_{j=0}^m H_j$, and endow $H_{\bf B}$ with the (semi-)norm $\|\cdot\|_{{\bf B},{\bf h}}$, where ${\bf h} = (h_0,\ldots,h_m)\in\RR_+^{m+1}$:
\[
	\|x\|_{{\bf B}, {\bf h}}^2 := \sum\limits_{j=0}^m h_j \sum\limits_{n\in M} |x_n|^2\cdot\left\|B_je_n\right\|_{H'}^2 = \sum\limits_{n\in M} |x_n|^2\cdot b_{n,\bf h},\quad x\in H_{\bf B},
\]
where
\[
	b_{{n},{\bf h}} := \sum\limits_{j=0}^m h_j\left\|B_j e_{n}\right\|_{H'}^2,\qquad n\in M.
\]
For convenience, set ${\bf 1} := (1,\ldots,1)\in \RR^{m+1}_+$. Clearly, norms $\|\cdot\|_{{\bf B},{\bf h}}$ are equivalent to $\|\cdot\|_{{\bf B}, {\bf 1}}$. However, if needed, we will supplement the notation of the space $H_{\bf B}$ with the subscript ${\bf h}$ to emphasize the norm considered. 

The next 
proposition provides constructive characterization of elements of $H_{\bf B}$, establishes embedding $H_{\bf B}\subset \mathcal{D}_{\bf B}$ and presents sufficient conditions for the coincidence of these sets.

\begin{lemma}
	\label{Bessel_inequality}
	Let $B_0,B_1,\ldots,B_m$ satisfy conditions {\rm (B1)}, {\rm (B2)} and~{\rm (B4)}. Then $H_{\bf B}\subset \mathcal{D}_{\bf B}$ and, for every $x\in H_{\bf B}$, 
	\[
		B_jx = \sum\limits_{{n}\in M} x_n\cdot B_je_{n},\qquad j=0,1,\ldots,m,
	\]
	\begin{equation}
	\label{prenorm_representation}
		\|x\|_{{\bf B}, {\bf h}}^2 = \sum\limits_{j=0}^m h_j\left\|B_jx\right\|_{H'}^2,\quad {\bf h}\in\mathbb{R}_+^{m+1}.
	\end{equation}
	Furthermore, if $B_je_n\in\mathcal{D}\left(B_j^\ast\right)$ (for the definition of the adjoint operator see, {\it e.g.}~\cite[\S7.1]{Yos_95}) for every $n\in M$ and $j=0,1,\ldots,m$, then $\mathcal{D}_{\bf B} = H_{\bf B}$.
\end{lemma}

\begin{proof}
	Choose any $j\in\{0,1,\ldots,m\}$ and $x\in H_j$. Consider the sequence of partial sums $\{S_Nx\}_{N=1}^{\infty}$. Clearly, $S_Nx\in\mathcal{D}\left(B_j\right)$ by condition~(B1) and $S_Nx\to x$ in $H$ as $N\to\infty$. Also, by condition~(B2), for $N_1,N_2\in\NN$, $N_1 < N_2$, 
	\begin{gather*}
		\displaystyle \left\|B_j(S_{N_2}x) - B_j(S_{N_1}x)\right\|_{H'}^2 = \displaystyle \left\|\sum\limits_{{n}\in M_{N_2}\setminus M_{N_1}} x_n\cdot B_je_{n}\right\|_{H'}^2 \\ 
		= \displaystyle \sum\limits_{{n}\in M_{N_2}\setminus M_{N_1}} \left|x_n\right|^2\left\|B_je_{n}\right\|_{H'}^2.
	\end{gather*}
	Since $x\in H_j$, the later implies that $\left\{B_j(S_Nx)\right\}_{N=1}^\infty$ is fundamental. Completeness of $H'$ yields existence of $y\in H'$ such that $B_j(S_Nx) \to y$ as $N\to\infty$. Taking into account that $B_j$ is closed we conclude that $x\in\mathcal{D}\left(B_j\right)$ and $y = B_jx$. As a result, $H_j\subset \mathcal{D}(B_j)$ and, hence, $H_{\bf B}\subset \mathcal{D}_{\bf B}$. Using straightforward calculations we see that equality~\eqref{prenorm_representation} holds true.
	
	Assume that $B_je_n\in \mathcal{D}\left(B_j^\ast\right)$, for every $n\in M$. Let us show that $\mathcal{D}\left(B_j\right) = H_j$. To this end we choose any $x\in \mathcal{D}\left(B_j\right)$. Since $B_j$ is densely defined, its adjoint operator $B_j^\ast$ is well defined~(see, {\it e.g.}~\cite[\S7.1]{Yos_95}). It is clear that, for $m,n\in M$, $m\ne n$, 
	\[
		0 = \left(B_je_n,B_je_m\right)_{H'} = \left(B_j^\ast B_j e_n, e_m\right)_H.
	\]
	Hence, there exists $\lambda_n\in\mathbb{C}$ such that $B_j^\ast B_j e_n = \lambda_n e_n$. Consequently, 
	\[
		\left\|B_je_n\right\|_{H'}^2 = \left(B_je_n,B_je_n\right)_{H'} = \left(B_j^\ast B_je_n,e_n\right)_H = \lambda_n.
	\] 
	We let $\varphi_n := \frac{B_je_n}{\left\|B_je_n\right\|_{H'}}$ and observe that the system $\left\{\varphi_n\right\}_{n\in M}$ is orthonormal in $H'$. Then by the Bessel inequality (see, {\it e.g.}~\cite[\S3.4]{Yos_95})
	\[
		\|B_jx\|_{H'}^2 \geqslant \sum\limits_{n\in M} \left|\left(B_jx, \varphi_n\right)_{H'}\right|^2 = \sum\limits_{n\in M} \left|\left(x, B_j^\ast\varphi_n\right)_{H}\right|^2 = \sum\limits_{n\in M} \left|x_n\right|^2 \left\|B_je_n\right\|_{H'}^2,
	\]
	and we conclude that $x\in H_j$. As a result, $\mathcal{D}_{\bf B} = H_{\bf B}$.
\end{proof}

Finally we present the conditions we will require intermediate operator to satisfy. Let $X$ be a normed space over $\mathbb{C}$ with the norm $\|\cdot\|_X$, $A:H\to X$ be a linear operator with domain $\mathcal{D}(A)$ and $f\in X^\ast$ be a linear bounded functional. We will require $A$ and $f$ to satisfy some of conditions below:
%
%
%
%
%
\begin{enumerate}
	\item [(A1)] $\forall n\in M\Rightarrow e_{n}\in \mathcal{D}(A)$;
	\item [(A2)] $\exists n_0\in M$ $\Rightarrow$ $Ae_{n_0}\ne 0$;
	\item [(A3)] $A$ is closable (see, {\it e.g.}~\cite[\S2.6]{Yos_95});
	\item [(Af1)] $\exists n_0\in M$ $\Rightarrow$ $\left<f,Ae_{n_0}\right>\ne 0$;
	\item [(Af2)] the functional $g_f(x) = \left<f,Ax\right>:H\to \mathbb{C}$ with domain $\mathcal{D}(A)$ is closable.
	
\end{enumerate}

Evidently, if $f\in\mathcal{D}(A^\ast)$ then the functional $g_f$ defined in condition~(Af2) is continuous and, hence, closable.

\section{Mean-squared Hardy-Littlewood-P\'olya and Taikov type inequalities}
\label{Sec:Taikov}

In this section we present a unified approach for obtaining sharp inequalities for the norms of linear operators in Hilbert space. We consider two cases: 
\begin{enumerate}
	\item The so-called {\it Taikov type inequality} that estimates the value $\left<f,Ax\right>$ of bounded functional $f\in X^\ast$ on the image $Ax$ of element $x$ under intermediate operator $A$ in terms of the norms of images $B_0 x,B_1 x,\ldots,B_m x$ of element $x$ under operators $B_0,B_1,\ldots,B_m$;
	\item The so-called {\it Hardy-Littlewood-P\'olya type inequality} that estimates the norm of the image $Ax$ of element $x$ under intermediate operator $A$ in terms of the norms of images $B_0 x, B_1 x,\ldots, B_m x$ of element $x$ under operators $B_0,B_1,\ldots,B_m$. 
\end{enumerate}
These inequalities can be considered as abstract versions of inequalities for the norms of derivatives. In addition, we obtain the additive version of the Taikov type inequality and solve the related Stechkin problem of the best approximation of operators by linear bounded ones. 

\subsection{Mean-squared Taikov type inequality}


\begin{theorem}
\label{Taikov_type_generic}
	Let $m\in\mathbb{Z}_+$, linear operators $B_0,B_1,\ldots,B_m$ satisfy conditions~{\rm (B1)}--~{\rm (B4)}, and linear operator $A$ and functional $f\in X^\ast$ satisfy conditions {\rm (A1)}, {\rm (Af1)} and {\rm (Af2)}. Let also 
	\begin{equation}
	\label{condition_0}
		\widetilde{\sum\limits_{{n}\in M}} \frac{\left|\left<f, Ae_{n}\right>\right|^2}{b_{{n},{\bf 1}}} < \infty.
	\end{equation}
	Then, for ${\bf h}\in \mathbb{R}^{m+1}_+$ and $x\in \mathcal{D}(A)\cap H_{\bf B}$, there holds sharp 
	inequality
	\begin{equation}
	\label{mean_squared_Taikov_type_inequality_generic}
		\displaystyle \left|\left<f, Ax\right>\right| \leqslant \displaystyle \left(\widetilde{\sum\limits_{{n}\in M}} \frac{\left|\left<f, Ae_{n}\right>\right|^2}{b_{{n},{\bf h}}}\right)^{\frac 12} \|x\|_{{\bf B},{\bf h}}.
	\end{equation}
	Moreover, the extremal sequence of elements in inequality~\eqref{mean_squared_Taikov_type_inequality_generic} is delivered by 
	\begin{equation}
	\label{extremal_Taikov_element}
		x_{{\bf h},N} := \widetilde{\sum\limits_{{n}\in M_N}} \frac{\overline{\left<f, Ae_{n}\right>}}{b_{{n}, {\bf h}}}\cdot e_{n},\qquad N\in\NN. 
	\end{equation}
\end{theorem}



\begin{rem}
	If we assume additionally that $g_f$ defined in condition~(Af2) is closed then inequality~\eqref{mean_squared_Taikov_type_inequality_generic} holds for every $x\in H_{\bf B}$.
\end{rem}

\begin{rem}
	Condition~\eqref{condition_0} implies $\left<f,Ae_{n}\right> = 0$ for every ${n}\in M$ such that $b_{{n},{\bf h}} = 0$.
\end{rem}

\begin{rem}
	Theorem~\ref{Taikov_type_generic} generalizes Theorem~1 in~\cite{BabLigShu_06}. The latter can be obtained by choosing $H' := H$, $h_0 := 1$ and $B_0 := {\rm id}_H$, where ${\rm id}_H:H\to H$ is the identity operator.
\end{rem}


\begin{proof}
	Let ${\bf h}\in\mathbb{R}_+^{m+1}$ and $x\in\mathcal{D}(A)\cap H_{\bf B}$. Condition~\eqref{condition_0} implies that the series $\widetilde{\sum\limits_{{n}\in M}} \frac{\left|\left<f, Ae_{n}\right>\right|^2}{b_{{n},{\bf h}}}$ is convergent. 
	Consider partial sums $\{S_Nx\}_{N=1}^\infty$. By conditions~(B1) and~(A1), $S_Nx\in \mathcal{D}(A)\cap H_{\bf B}$, and using the Schwarz inequality (see, {\it e.g.}~\cite[\S1.5]{Yos_95}) we obtain
	\begin{gather*}
		\left|\left<f,A(S_Nx)\right>\right| \!=\! \displaystyle \left|\sum\limits_{{n} \in M_N} \!\!\!\left<f, Ae_{n}\right> \cdot x_n\right| \leqslant \left(\widetilde{\sum\limits_{{n}\in M_N}} \frac{\left|\left<f, Ae_{n}\right>\right|^2}{b_{{n},{\bf h}}}\right)^{\frac 12}\!\!\!\left(\sum\limits_{{n}\in M_N}b_{{n},{\bf h}}\left|x_n\right|^2\right)^{\frac 12}\\ 
		\leqslant \displaystyle \left(\widetilde{\sum\limits_{{n}\in M}} \frac{\left|\left<f, Ae_{n}\right>\right|^2}{b_{{n}, {\bf h}}} \right)^{\frac 12} \|x\|_{{\bf B},{\bf h}}.
	\end{gather*}
	Applying similar arguments we can show that the sequence $\left\{\left<f,A(S_Nx)\right>\right\}_{N=1}^\infty$ is also fundamental and, hence, convergent. By condition~(Af2), the functional $g_f:H\to \mathbb{C}$ is closable and, since $S_Nx\to x$ in $H$ as $N\to\infty$, we conclude that $\left<f,A(S_Nx)\right> \to \left<f,Ax\right>$ as $N\to\infty$. Taking the limit in the left hand part of above inequality, we prove inequality~\eqref{mean_squared_Taikov_type_inequality_generic}. 
	
	Let us show sharpness of inequality~\eqref{mean_squared_Taikov_type_inequality_generic}. Indeed,
	\[
		\left<f, Ax_{{\bf h},N}\right> = \widetilde{\sum\limits_{{n}\in M_N}}\frac{\left|\left<f,Ae_{n}\right>\right|^2}{b_{{n},{\bf h}}},
	\]
	\[
		\left\|x_{{\bf h},N}\right\|_{{\bf B},{\bf h}}^2 = 
		\sum\limits_{j=0}^m h_j\widetilde{\sum\limits_{{n}\in M_N}} \frac{\left|\left<f,Ae_{n}\right>\right|^2}{b^2_{{n},{\bf h}}}\cdot \left\|B_je_{n}\right\|_{H'}^2 = \widetilde{\sum\limits_{{n}\in M_N}} \frac{\left|\left<f,Ae_{n}\right>\right|^2}{b_{{n},{\bf h}}}.
	\]
	By~(B3) and~(Af1), $\left\|x_{{\bf h},N}\right\|_{{\bf B}, \bf h}\ne 0$ for sufficiently large $N\in\mathbb{N}$. Then
	\[
		\frac{\left|\left<f,Ax_{{\bf h},N}\right>\right|^2}{\left\|x_{{\bf h}, N}\right\|^2_{{\bf B},{\bf h}}
		} = \widetilde{\sum\limits_{{n}\in M_N}} \frac{\left|\left<f,Ae_{n}\right>\right|^2}{b_{{n},{\bf h}}}\to \widetilde{\sum\limits_{{n}\in M}} \frac{\left|\left<f,Ae_{n}\right>\right|^2}{b_{{n},{\bf h}}},\quad \text{as }N\to\infty,
	\]
	which proves that inequality~\eqref{mean_squared_Taikov_type_inequality_generic} is sharp. 
\end{proof}

\subsection{Additive Taikov type inequality}

Let $p, q\in \ZZ_+$ and consider operators $C_j:H\to H'$, $j=0,1,\ldots,p$, and $D_k:H\to H'$, $k=0,1,\ldots,q$. In this subsection we establish sharp additive Taikov type inequality estimating the value of functional $f\circ A$ (defined in condition~(Af2)) on elements $x\in H$ in terms of norms $\|x\|_{{\bf C},{\bf h'}}$ and $\|x\|_{{\bf D}, {\bf h''}}$, where ${\bf h'}\in\mathbb{R}_+^{p+1}$, ${\bf h''}\in\mathbb{R}_+^{q+1}$. Simultaneously we solve the Stechkin problem on the best approximation of $f\circ A : H_{{\bf C},{\bf h'}}\to \mathbb{C}$ by linear bounded functionals on the class 
\[
	W_{{\bf D},{\bf h''}} := \left\{x\in \mathcal{D}(A)\cap H_{\bf C}\cap H_{\bf D}\,:\,\|x\|_{{\bf D},{\bf h''}} 
	\leqslant 1\right\}.
\]

For convenience, throughout this subsection we set $b_{n,{\bf h}} = c_{n, {\bf h'}} + d_{n ,{\bf h''}}$, $n\in M$, and $H_{\bf B} := H_{\bf C}\cap H_{\bf D}$.

\begin{theorem}
\label{Taikov_type}
	Let $p,q\in\mathbb{Z}_+$, both operators $C_0,\ldots,C_{p}$ and $D_0,\ldots,D_{q}$ satisfy conditions~{\rm (B1)}--{\rm (B4)}, operator $A$ and functional $f\in X^\ast$ satisfy conditions~{\rm (A1)}, {\rm (Af1)},~{\rm (Af2)}. Let also inequality~\eqref{condition_0} be fulfilled, {\it i.e.}, 
	\[
	    \widetilde{\sum\limits_{n\in M}}\frac{\left|\left<f,Ae_n\right>\right|^2}{b_{n,{\bf 1}}} < \infty.
	\]
	Then, for ${\bf h'}\in \mathbb{R}^{p+1}_+$, ${\bf h''}\in\mathbb{R}_+^{q+1}$ and $x\in \mathcal{D}(A)\cap H_{\bf B}$, there holds sharp inequality
	\begin{equation}
	\label{Taikov_type_inequality}
	    \begin{array}{rcl}
		\displaystyle \left|\left<f, Ax\right>\right| & \leqslant & \displaystyle \left(\widetilde{\sum\limits_{{n}\in M}} \frac{\left|\left<f, Ae_{n}\right>\right|^2 c_{n,{\bf h'}}}{b_{{n},{\bf h}}^2}\right)^{\frac 12} \left\|x\right\|_{{\bf C},{\bf h'}} \\ 
		& & \quad \displaystyle +  \left(\widetilde{\sum\limits_{{n}\in M}} \frac{\left|\left<f, Ae_{n}\right>\right|^2 d_{{n},{\bf h''}}}{b_{{n},{\bf h}}^2}\right)^{\frac 12}\|x\|_{{\bf D}, {\bf h''}}.
		\end{array}
	\end{equation}
	The sequence $\left\{x_{{\bf h},N}\right\}_{N=1}^\infty$ defined by~\eqref{extremal_Taikov_element} is extremal in~\eqref{Taikov_type_inequality}. 
\end{theorem}


\begin{rem}
	If we additionally assume that $g_f$ is closed then inequality~\eqref{Taikov_type_inequality} holds for every $x\in H_{\bf B}$.
\end{rem}

Now, following~\cite{Ste_67} (see also~\cite{AreGab_95,Are_96,BKKP}) we recall the statement of the Stechkin problem on the best approximation of operators by linear bounded ones. Let $X,Y$ be Banach spaces, $A:X\to Y$ be operator, not necessarily linear, with domain $\mathcal{D}(A)$, $W\subset \mathcal{D}(A)$ be some class. For a linear bounded operator $S:X\to Y$, denote the error of approximation of $A$ by $S$ on the class $W$:
\[
	U\left(A,S;W\right) := \sup\limits_{x\in W} \|Ax - Sx\|_Y.
\]
Let $N>0$ and $\mathcal{L}(N)$ be the space of linear bounded operators $S:X \to Y$ whose norm is bounded by $N$. 

The Stechkin problem consists of finding the error of {\it the best approximation of operator $A$} by linear bounded operators on the class $W$: 
\begin{equation}
\label{Stechkin_problem}
	E_N\left(A;W\right) = \inf\limits_{S\in\mathcal{L}(N)} \sup\limits_{x\in W} \left\|Ax - Sx\right\|_{Y},
\end{equation}
and finding extremal operators $S^*\in\mathcal{L}(N)$ (if any exists) for which the infimum in the right hand part of~\eqref{Stechkin_problem} is achieved.

In~\cite[\S2]{Ste_67} S.\,B.~Stechkin established simple yet powerful lower bound for~\eqref{Stechkin_problem} in terms of modulus of continuity of operator $A$ on class $W$:
\[
	\Omega\left(\delta;A;W\right) := \left\{\|Ax\|_Y\,:\,x\in W,\,\|x\|_X\leqslant \delta\right\},\quad \delta\geqslant 0.
\]
\begin{theorem}
\label{Stechkin_lemma}
	Let $A$ be a homogeneous operator, $W\subset\mathcal{D}(A)$ be a centrally-symmetric convex set. Then for every $N\geqslant 0$ and $\delta\geqslant 0$,
	\begin{equation}
	\label{Stechkin_lower_estimate}
		E_N(A;W) \geqslant \sup\limits_{\delta\geqslant 0} \left(\Omega(\delta;A;W) - N\delta\right) = \sup\limits_{x\in W} \left(\|Ax\|_Y - N\|x\|_X\right).
	\end{equation} 
	Moreover, if there exist an element $x_0\in W$ and a linear bounded operator $S_0:X\to Y$ such that $\|Ax_0\|_Y = U\left(A,S_0;W\right) + \|S_0\|\cdot\|x_0\|_X$ then $\Omega\left(\|x_0\|_X;A;W\right) = \|Ax_0\|_Y$,
	\[
		E_{\|S_0\|}(A;W) = U\left(A,S_0;W\right) = \|Ax_0\|_Y - \|S_0\|\cdot \|x_0\|_X,
	\]
	and $S_0$ is extremal operator in problem~\eqref{Stechkin_problem} for $N=\|x_0\|_X$.
\end{theorem}

For a good overview of known results on the Stechkin problem and discussion of related problems we refer the reader to surveys~\cite{AreGab_95,Are_96} and book~\cite{BKKP}.

Let us formulate the results on the solution to problem~\eqref{Stechkin_problem} for functional $f\circ A:H_{{\bf C}, {\bf h'}}\to\mathbb{C}$ on the class $W_{{\bf D},{\bf h''}}$. First, we introduce the set of functionals $G_\mu:H_{{\bf C},{\bf h}'} \to\mathbb{C}$, $\mu\geqslant 0$, that will be extremal in this problem:
\[
	G_\mu x := \widetilde{\sum\limits_{{n}\in M}} \frac{\left<f, Ae_{n}\right> c_{{n}, {\bf h'}}}{c_{{n}, {\bf h'}} + \mu d_{{n}, {\bf h''}}}\cdot x_n,\qquad x\in H_{{\bf C}, {\bf h}'}.
\]
Providing that condition~\eqref{condition_0} is fulfilled, it is evident that, for every $\mu > 0$,
\[
	\left\|G_\mu\right\| = \|G_\mu\|_{H_{{\bf C},{\bf h}'}\to \mathbb{C}} = \widetilde{\sum\limits_{{n}\in M}} \frac{\left|\left<f, Ae_{n}\right>\right|^2c_{{n},{\bf h}'}}{\left(c_{{n},{\bf h}'} + \mu d_{{n},{\bf h}''}\right)^2},
\]
and $\left\|G_\mu\right\|$ is non-increasing and continuous on $\RR_+$. By the monotone convergence theorem, 
\[
	\lim\limits_{\mu\to+\infty}\left\|G_\mu\right\| = \lim\limits_{\mu\to+\infty} \widetilde{\sum\limits_{{n}\in M}} \frac{\left|\left<f, Ae_{n}\right>\right|^2c_{{n},{\bf h}'}}{\left(c_{{n},{\bf h}'} + \mu d_{{n},{\bf h}''}\right)^2} = \sum\limits_{n\in M^*} \frac{\left|\left<f,Ae_n\right>\right|^2}{c_{n,{\bf h}'}} =: N^*, 
\]
where $M^* := \left\{n\in M\,:\,c_{n,{\bf h}'}\ne 0\text{ and }d_{n,{\bf h}''}=0\right\}$, and
\[
	\lim\limits_{\mu\to 0^+} \left\|G_\mu\right\| = \widetilde{\sum\limits_{{n}\in M}} \frac{\left|\left<f,Ae_{n}\right>\right|^2}{c_{{n},{\bf h}'}}.
\]
Hence, either $\lim\limits_{\mu\to 0^+}\left\|G_\mu\right\| = \left\|G_0\right\|$ when $G_0$ is bounded, or $\lim\limits_{\mu\to 0^+}\left\|G_\mu\right\| = +\infty$ when $G_0$ is unbounded. 



\begin{theorem}
\label{Stechkin_theorem}
	Let $p,q\in\ZZ_+$, both operators $C_0,\ldots,C_{p}$ and $D_0,\ldots,D_{q}$ satisfy conditions~{\rm (B1)}--{\rm (B4)}, operator $A$ and functional $f\in X^\ast$ satisfy conditions~{\rm (A1)}, {\rm (Af1)}, {\rm (Af2)}. Also, let inequality~\eqref{condition_0} be fulfilled, {\it i.e.}
	\[
	    \widetilde{\sum\limits_{n\in M}}\frac{\left|\left<f,Ae_n\right>\right|^2}{b_{n,{\bf 1}}} < \infty
	\]
	and either $N \in \left(N^*,+\infty\right)$ if $f\circ A : H_{{\bf C},{\bf h}'}\to\mathbb{C}$ is unbounded or $N\in \left(N^*,\left\|f\circ A\right\|\right]$ if $f\circ A:H_{{\bf C},{\bf h}'}\to\mathbb{C}$ is bounded, and $\mu \geqslant 0$ be such that $N = \left\|G_\mu\right\|$. Then
	\begin{equation}
	\label{Stechkin_result}
		E_N\left(f\circ A; W_{{\bf D}, {\bf h''}}\right) = \mu\cdot \left(\widetilde{\sum\limits_{{n}\in M}} \frac{\left|\left<f, Ae_{n}\right>\right|^2 d_{{n}, {\bf h''}}}{\left(c_{{n}, {\bf h'}} + \mu d_{{n}, {\bf h''}}\right)^2}\right)^{\frac 12},
	\end{equation}
	and the functional $G_\mu$ is extremal in problem~\eqref{Stechkin_problem}. 
\end{theorem}

Note that in~\cite{Bab_Bil_12} the Stechkin problem was solved in a similar setting, when an unbounded functional, represented as a composition of a functional and a power of self-adjoint operator in Hilbert space, is approximated by bounded functionals on the set defined with the help of higher order power of this self-adjoint operator.

\begin{rem}
    Theorem~\ref{Stechkin_theorem} allows solving closely related problems of finding the modulus of continuity of $f\circ A$ on the set $W_{{\bf D},{\bf h}''}$, and the problem of the best recovery of $f\circ A$ on the set $W_{{\bf D}, {\bf h}''}$, whose elements are given with an error (see, {\it e.g.},~\cite{Are_96} or~\cite[\S7]{BKKP}). 
\end{rem}


\begin{proof}[The proof of Theorems~\ref{Taikov_type} and~\ref{Stechkin_theorem}]
	By the Schwarz inequality, for $x\in H_{{\bf C},{\bf h}'}$:
	\[
		\begin{array}{rcl}
			\displaystyle \left|G_{\mu} x\right| & \leqslant &\displaystyle  \|G_\mu\|\cdot \left(\widetilde{\sum\limits_{{n}\in M}} c_{{n},{\bf h'}}\left|x_n\right|^2\right)^{\frac 12} = \displaystyle \left(\widetilde{\sum\limits_{n\in M}}\frac{\left|\left<f,Ae_n\right>\right|^2}{c_{n,{\bf h}'}}\right)^{\frac 12} \left\|x\right\|_{{\bf C}, {\bf h'}}.
		\end{array}
	\]
	For $x\in H_{\bf B}\cap \mathcal{D}(A)$, we consider partial sums $\{S_Nx\}_{N=1}^\infty$. Evidently, $S_Nx\to x$ in $H$ and in $H_{{\bf C},{\bf h}'}$ as $N\to\infty$. Using the Schwarz inequality, closability of functional $g_f:H\to\mathbb{C}$ and boundedness of functional $G_\mu$, we obtain
	\begin{equation}
	\label{difference_estimate}
		\left|\left<f, Ax\right> - G_{\mu} x\right| \leqslant \left(\widetilde{\sum\limits_{{n}\in M}} \frac{\left|\left<f, Ae_{n}\right>\right|^2\cdot\mu^2 d_{{n}, {\bf h''}}}{\left(c_{{n}, {\bf h'}} + \mu d_{{n}, {\bf h''}}\right)^2}\right)^{\frac 12} \|x\|_{{\bf D}, {\bf h''}}.
	\end{equation}
	Combining above estimates in the case $\mu = 1$ and applying the triangle inequality, we establish the desired inequality~(\ref{Taikov_type_inequality}). It is not difficult to verify that the sequence $\left\{x_{{\bf h},N}\right\}_{N=1}^\infty$ defined by~\eqref{extremal_Taikov_element} is extremal in  inequality~\eqref{Taikov_type_inequality}, which finishes the proof of Theorem~\ref{Taikov_type}. 
	
	It remains to prove Theorem~\ref{Stechkin_theorem}. Let $N = \left\|G_\mu\right\|$ with some $\mu\geqslant 0$. By~\eqref{difference_estimate},
	\[
		E_N\left(f\circ A; W_{{\bf D}, {\bf h''}}\right)\leqslant \sup\limits_{x\in W_{{\bf D},{\bf h}''}} \left|\left<f,Ax\right> - G_\mu x\right|\leqslant \mu\cdot \left(\widetilde{\sum\limits_{{n}\in M}} \frac{\left|\left<f, Ae_{n}\right>\right|^2 d_{{\bf n}, {\bf h''}}}{\left(c_{{n}, {\bf h'}} + \mu d_{{n}, {\bf h''}}\right)^2}\right)^{\frac 12}.
	\]
	Let us show that the later inequality is sharp. To this end, for $L\in\NN$, we consider the element 
	\[
		x^\mu_{{\bf h},L} := \widetilde{\sum\limits_{{n} \in M_L}} \frac{\overline{\left<f, Ae_{n}\right>}}{c_{{n}, {\bf h'}} + \mu d_{{n}, {\bf h''}}}\cdot e_{n}.
	\]
	Evidently,
	\[
		\left<f, Ax^\mu_{{\bf h},L}\right> = \widetilde{\sum\limits_{{n}\in M_L}} \frac{\left|\left<f, Ae_{n}\right>\right|^2}{c_{{n}, {\bf h'}} + \mu d_{{n}, {\bf h''}}},
	\]
	\[
		\left\|x^\mu_{{\bf h},L}\right\|_{{\bf C}, {\bf h'}}^2 \!=\! \widetilde{\sum\limits_{{n}\in M_L}} \frac{\left|\left<f, Ae_{n}\right>\right|^2 c_{{n}, {\bf h'}}}{\left(c_{{n}, {\bf h'}} + \mu d_{{n}, {\bf h''}}\right)^2} \;\;{\rm and}\;\; \left\|x^\mu_{{\bf h},L}\right\|_{{\bf D}, {\bf h''}}^2 \!=\! \widetilde{\sum\limits_{{n}\in M_L}} \frac{\left|\left<f, Ae_{n}\right>\right|^2 d_{{n}, {\bf h''}}}{\left(c_{{n}, {\bf h'}} + \mu d_{{n}, {\bf h''}}\right)^2}.
	\]
	Using the lower estimate~\eqref{Stechkin_lower_estimate}, we obtain
	\[
		\begin{array}{rcl}
			\displaystyle E_N\left(f\circ A; W_{{\bf D}, {\bf h''}}\right) & \geqslant & \displaystyle \lim\limits_{L\to\infty}\frac{\left|\left<f, Ax^\mu_{{\bf h},L}\right>\right| - N\cdot\left\|x^\mu_{{\bf h},L}\right\|_{{\bf C}, {\bf h'}}}{\left\|x^\mu_{{\bf h},L}\right\|_{{\bf D}, {\bf h''}}} \\ 
			& = & \displaystyle \lim\limits_{L\to\infty}\left(\frac{1}{\left\|x^\mu_{{\bf h},L}\right\|_{{\bf D}, {\bf h''}}} \widetilde{\sum\limits_{{n}\in M_L}} \frac{\left|\left<f, Ae_{n}\right>\right|^2 \cdot\mu d_{{n}, {\bf h''}}}{\left(c_{{n}, {\bf h'}} + \mu d_{{n}, {\bf h''}}\right)^2}\right) \\
			& = & \displaystyle \mu \left(\widetilde{\sum\limits_{{n}\in M}} \frac{\left|\left<f, Ae_{n}\right>\right|^2 d_{{n}, {\bf h''}}}{\left(c_{{n}, {\bf h'}} + \mu d_{{n}, {\bf h''}}\right)^2}\right)^\frac{1}{2},
		\end{array}
	\]
	which finishes the proof. Remark that in the above relations we skipped those $L$'s for which $\left\|x^\mu_{{\bf h},L}\right\|_{{\bf D},{\bf h}''} = 0$.
	
	Let us show that $\|G_\mu\|$ attains all values in $(N^*,\|f\circ A\|]$ when $f\circ A$ is bounded, and all values in $(N^*,+\infty)$ when $f\circ A$ is unbounded. Indeed, assume first that $f\circ A:H_{{\bf C},{\bf h}'}\to\mathbb{C}$ is bounded functional. Clearly, $\left<f,Ax\right> = G_0x$ on dense in $H$ set $\mathcal{F} := {\rm span}\left\{e_n\,:\,n\in M\right\}$. Hence, $\|f\circ A\| \geqslant \left\|G_0\right\|$. On the other hand, by Theorem~\ref{Taikov_type_generic},
	\[
		\|f\circ A\| \leqslant \widetilde{\sum\limits_{n\in M}}\frac{\left|\left<f, Ae_n\right>\right|^2}{c_{n,{\bf h'}}} = \|G_0\|.
	\] 
	Hence, $\|f\circ A\| = \left\|G_0\right\|$ and $\left\|G_\mu\right\|$ attains all values in $\left(N^*,\left\|f\circ A\right\|\right]$ when $\mu\in[0,+\infty)$.
	
	Finally, we assume that $f\circ A:H_{{\bf B}',{\bf h}'}\to\mathbb{C}$ is unbounded. According to Theorem~\ref{Taikov_type_generic} this implies that the series 
	\[
		\widetilde{\sum\limits_{n\in M}}\frac{\left|\left<f, Ae_n\right>\right|^2}{c_{n,{\bf h'}}}
	\]
	is divergent. Hence, $\lim\limits_{\mu\to 0^+}\left\|G_\mu\right\| = +\infty$ and the norm $\|G_\mu\|$ decreases and attains all values in $\left(N^*,+\infty\right)$ as $\mu\in(0,+\infty)$.
\end{proof}

\begin{rem}
	Using the arguments similar to those applied to prove Theorem~\ref{Stechkin_theorem}, we can show that $E_N\left(f\circ A; W_{{\bf D},{\bf h}''}\right) = +\infty$ when $N \in \left[0, N^*\right)$, and 
	\[
		E_{N^*}\left(f\circ A; W_{{\bf D},{\bf h}''}\right) = \widetilde{\sum\limits_{n\in M}} \frac{\left|\left<f,Ae_n\right>\right|^2}{d_{n,{\bf h}''}},
	\]
	with the functional $G_{+\infty}x := \sum\limits_{n\in M^*}\left<f,Ae_n\right>\cdot x_n$, $x\in H_{\bf C, h'}$, being extremal in Problem~\ref{Stechkin_problem} providing that the right hand part of the above equality if finite.
\end{rem}

\subsection{Hardy-Littlewood-P\'olya type inequality}

Let $X = H'$ so that $A$ acts from $H$ to $H'$. Taking $\sup$ in~\eqref{mean_squared_Taikov_type_inequality_generic} over all functionals $f\in X^\ast \cong H'$ such that $\left\|f\right\|_{H'} \leqslant 1$ we obtain the following result. 

\begin{theorem}
\label{HLP}
	Let $m\in\mathbb{Z}_+$, operators $B_0,B_1,\ldots,B_m$ satisfy conditions~{\rm (B1)}--{\rm (B4)}, operator $A$ satisfy conditions~{\rm (A1)}--{\rm (A3)}, and
	\begin{equation}
	\label{condition_2}
		\sup\limits_{f\in H':\atop \|f\|_{H'}\leqslant 1} \widetilde{\sum\limits_{{n}\in M}} \frac{\left|\left(f,Ae_{n}\right)_{H'}\right|^2}{b_{{n}, {\bf 1}}} < \infty.
	\end{equation}
	Then, for ${\bf h} \in \mathbb{R}_+^{m+1}$ and $x\in \mathcal{D}(A)\cap H_{\bf B}$, there holds true sharp inequality
	\begin{equation}
	\label{HLP_type_inequality}
		\left\|Ax\right\|_{H'} \leqslant \left(\sup\limits_{f\in H':\atop\|f\|_{H'}\leqslant 1} \widetilde{\sum\limits_{{n}\in M}} \frac{\left|\left(f, Ae_{n}\right)_{H'}\right|^2}{b_{{n},{\bf h}}}\right)^{\frac{1}{2}}\left\|x\right\|_{{\bf B},{\bf h}}.
	\end{equation}
\end{theorem}

\begin{rem}
	If $b_{{n},{\bf 1}} = 0$ for some ${n}\in M$ then condition~\eqref{condition_2} implies $Ae_{n} = 0$.
\end{rem}

\begin{rem}
	If $A$ is closed then inequality~\eqref{HLP_type_inequality} holds true for every $x\in H_{\bf B}$.
\end{rem}

\begin{rem}
	 Theorem~\ref{HLP} generalizes Hardy-Littlewood-P\'olya type inequality established in~\cite{BabLigShu_06} (see the last inequality in~\cite{BabLigShu_06}), which can be obtained by setting $H' := H$, $h_0 := 1$ and $B_0 := {\rm id}_{H}$. 
\end{rem}

\begin{proof}
	Since condition~(Af2) holds true for every $g\in\mathcal{D}\left(A^\ast\right)$, by Theorem~\ref{Taikov_type} and condition~\eqref{condition_2}, we have that, for every $g\in\mathcal{D}\left(A^\ast\right)$, $\|g\|_{H'}\leqslant 1$, and $x\in\mathcal{D}(A)\cap H_{\bf B}$,
	\begin{equation}
	\label{intermediate_ineq}
		\left|\left(g,Ax\right)_{H'}\right| \leqslant \left(\sup\limits_{f\in H':\atop\|f\|_{H'}\leqslant 1} \widetilde{\sum\limits_{{n}\in M}} \frac{\left|\left(f, Ae_{n}\right)_{H'}\right|^2}{b_{{n},{\bf h}}}\right)^{\frac{1}{2}}\left\|x\right\|_{{\bf B},{\bf h}}.
	\end{equation}
	It is well-known~\cite[Chap.~XII]{DanSwa_63} that $\mathcal{D}\left(A^\ast\right)$ is dense in $H'$. Hence, inequality~\eqref{intermediate_ineq} holds true for a set of functionals dense in the unit ball $U_{H'} = \left\{f\in H'\,:\,\|f\|_{H'}\leqslant 1\right\}$ in ${H'}$. Let us consider any $g\in U_{H'}$. There exists a sequence $U_{H'}\cap\mathcal{D}\left(A^\ast\right)\supset\left\{g_N\right\}_{N=1}^\infty\to g$. Since $Ax$ defines a continuous functional on $H'$, we conclude that~\eqref{intermediate_ineq} holds true for every $g\in U_{H'}$. Taking the sup in the left hand part of~\eqref{intermediate_ineq} over $g\in U_{H'}$, we obtain the desired inequality~\eqref{HLP_type_inequality}.
	
	Let us demonstrate that inequality~\eqref{HLP_type_inequality} is sharp. To this end we let $\left\{f_\varepsilon\right\}_{\varepsilon > 0}\subset U_{H'}$ be a set of elements such that, for $\varepsilon > 0$,
	\[
		\widetilde{\sum\limits_{{n}\in M}} \frac{\left|\left(f_\varepsilon,Ae_{n}\right)_{H'}\right|^2}{b_{{n},{\bf h}}} > \sup\limits_{f\in H'\,:\atop\|f\|_{H'}\leqslant 1}\widetilde{\sum\limits_{{n}\in M}} \frac{\left|\left(f,Ae_{n}\right)_{H'}\right|^2}{b_{{n},{\bf h}}} - \frac{\varepsilon}2.
	\]
	For $N\in\mathbb{N}$, let us consider the set $\left\{x_{{\bf h},N}^\varepsilon\right\}_{\varepsilon > 0}\subset H$ defined as follows
	\[
		x_{{\bf h},N}^\varepsilon := \widetilde{\sum\limits_{{n}\in M_N}} \frac{\overline{\left(f_\varepsilon,Ae_{n}\right)_{H'}}}{b_{{n},{\bf h}}}\cdot e_{n}.
	\]
	Then
	\[
		\left\|Ax_{{\bf h},N}^\varepsilon\right\|_{H'} \geqslant \left(f_\varepsilon, Ax^\varepsilon_{{\bf h},N}\right)_{H'} = \widetilde{\sum\limits_{{n}\in M_N}} \frac{\left|\left(f_\varepsilon,Ae_{n}\right)_{H'}\right|^2}{b_{{n},{\bf h}}},
	\]
	\[
		\left\|x_{{\bf h},N}^\varepsilon\right\|_{{\bf B},{\bf h}}^2 = \sum\limits_{j=0}^m h_j\left\|B_jx^\varepsilon_{{\bf h},N}\right\|_{H'}^2 = \widetilde{\sum\limits_{{n}\in M_N}} \frac{\left|\left(f_\varepsilon,Ae_{n}\right)_{H'}\right|^2}{b_{{n},{\bf h}}}.
	\]
	Hence, by conditions~(B3) and~(A2), for every sufficiently large $N$,
	\[
		\begin{array}{rcl}
			\displaystyle \frac{\left\|Ax_{{\bf h},N}^\varepsilon\right\|_{H'}^2}{\left\|x_{{\bf h},N}^\varepsilon\right\|_{{\bf B},{\bf h}}^2 
			} & \geqslant & \displaystyle \widetilde{\sum\limits_{{n}\in M_N}} \frac{\left|\left(f_\varepsilon,Ae_{n}\right)_{H'}\right|^2}{b_{{n},{\bf h}}} \geqslant \widetilde{\sum\limits_{{n}\in M}} \frac{\left|\left(f_\varepsilon,Ae_{n}\right)_{H'}\right|^2}{b_{{n},{\bf h}}} - \frac{\varepsilon}{2} \\
			& \geqslant & \displaystyle \sup\limits_{f\in H'\,:\atop\|f\|_{H'}\leqslant 1}\widetilde{\sum\limits_{{n}\in M}} \frac{\left|\left(f,Ae_{n}\right)_{H'}\right|^2}{b_{{n},{\bf h}}} - \varepsilon.
		\end{array}
	\]
	Letting $\varepsilon \to 0^+$, we see that inequality~\eqref{HLP_type_inequality} is sharp.
\end{proof}


\section{Multiplicative Hardy-Littlewood-P\'olya and Taikov type inequalities}
\label{Sec:Multiplicative}

In this section we obtain multiplicative Taikov and Hardy-Littlewood-P\'olya inequalities as consequences from their mean-squared analogues given by Theorems~\ref{Taikov_type_generic} and~\ref{HLP}. 

\subsection{Multiplicative Taikov type inequality}

Following the ideas in~\cite{Sha_90} we can establish the following result. 

\begin{theorem}
	\label{Taikov_for_multipliers_abst}
	Let $m\in\mathbb{Z}_+$, operators $B_0,\ldots,B_m$ satisfy conditions {\rm (B1)}--{\rm (B4)}, operator $A$ and functional $f\in X^\ast$ satisfy conditions~{\rm (A1)},~{\rm (Af1)},~{\rm (Af2)}. Also, let $\lambda_0,\ldots,\lambda_m > 0$ be such that $\lambda_0+\ldots+\lambda_m = 1$ and
	\begin{equation}
	\label{condition_constant_new_abst}
		\mathcal{C}(f,A,{\bf B},\overline{\lambda}) = \mathcal{C} := \sup\limits_{{\bf h}\in\RR_+^{m+1}}  \prod\limits_{j=0}^m h_j^{\lambda_j}\widetilde{\sum\limits_{n\in M}} \frac{\left|\left<f, Ae_{n}\right>\right|^2}{b_{n,\bf h}
		} < \infty.
	\end{equation}
	Then, for every $x\in \mathcal{D}(A)\cap H_{\bf B}$, there holds true sharp inequality
	\begin{equation}
	\label{Taikov_inequality_for_multipliers_abst}
		\displaystyle \left|\left<f, Ax\right>\right| \leqslant \sqrt{\mathcal{C}\cdot\prod\limits_{j=0}^m\lambda_j^{-\lambda_j}}\cdot\prod\limits_{j=0}^m \left\|B_jx\right\|_{H'}^{\lambda_j}.
	\end{equation}
\end{theorem}

\begin{proof}
	Observe that fulfilment of condition~\eqref{condition_0} follows from condition~\eqref{condition_constant_new_abst}. Applying Theorem~\ref{Taikov_type_generic}, for every $x\in\mathcal{D}(A)\cap H_{\bf B}$ and ${\bf h}\in\mathbb{R}_+^{m+1}$, we obtain
	\[
		\left|\left<f,Ax\right>\right|^2 \!\leqslant\! \left(\widetilde{\sum\limits_{n\in M}}\frac{\left|\left<f,Ae_{n}\right>\right|^2}{b_{n,{\bf h}}}\right)\cdot\left(\sum\limits_{j=0}^m h_j\!\left\|B_jx\right\|_{H'}^2\right) \!\leqslant\! \mathcal{C}\cdot\prod\limits_{j=0}^m \!h_j^{-\lambda_j}\!\sum\limits_{j=0}^m h_j\!\left\|B_jx\right\|_{H'}^2.
	\]
	Using the weighted AM-GM inequality~\cite[pp.~74-75]{Cve_12}, we minimize the right hand part of the above inequality
	\[
		\prod\limits_{j=0}^m h_j^{-\lambda_j}\cdot \sum\limits_{j=0}^m h_j \left\|B_jx\right\|_{H'}^2 \!=\! \displaystyle \prod\limits_{j=0}^m h_j^{-\lambda_j} \cdot \sum\limits_{j=0}^m \lambda_j\cdot\frac{h_j \left\|B_jx\right\|_{H'}^2}{\lambda_j} \!\geqslant\! \prod\limits_{j=0}^m \lambda_j^{-\lambda_j} \cdot\prod\limits_{j=0}^m \left\|B_jx\right\|_{H'}^{2\lambda_j}\!.
	\]
	Hence, inequality~\eqref{Taikov_inequality_for_multipliers_abst} is proved. 
	
	Let us show that inequality~\eqref{Taikov_inequality_for_multipliers_abst} is sharp. To this end, consider elements
	\[
		x^\ast_{{\bf h},N} := \widetilde{\sum\limits_{n\in M_N}} \frac{\overline{\left<f,Ae_{n}\right>}}{b_{n,\bf h}
		}\cdot e_{n},\qquad N\in\mathbb{N}.
	\]
	Straightforward calculations show that
	\begin{gather*}
		\left\|B_rx^\ast_{{\bf h},N}\right\|_{H'}^2 = \widetilde{\sum\limits_{n\in M_N}} \frac{\left|\left<f,Ae_{n}\right>\right|^2\cdot\left\|B_re_n\right\|_{H'}^2}{b_{n,\bf h}^2
		},\qquad r=0,1,\ldots,m,\\
		\left|\left<f, Ax^\ast_{{\bf h},N}\right>\right| = \widetilde{\sum\limits_{n\in M_N}} \frac{\left|\left<f,Ae_{n}\right>\right|^2}{b_{n,\bf h}
		} = \sum\limits_{r=0}^m h_r \left\|B_rx^\ast_{{\bf h},N}\right\|_{H'}^2.
	\end{gather*}
	By the weighted AM-GM inequality, for sufficiently large $N$, ensuring the denominator is positive, 
	\[	
		\displaystyle \frac{\left|\left<f, Ax^\ast_{{\bf h},N}\right>\right|}{\prod\limits_{j=0}^m \left\|B_jx^\ast_{{\bf h},N}\right\|_{H'}^{2\lambda_j}} \geqslant \prod\limits_{j=0}^m \left(\frac{h_j}{\lambda_j}\right)^{\lambda_j}.
	\]
	Multiplying both parts of above inequality by $\left|\left<f, Ax^\ast_{{\bf h},N}\right>\right|$, we obtain
	\[
		\displaystyle \frac{\left|\left<f, Ax^\ast_{{\bf h},N}\right>\right|^2}{\prod\limits_{j=0}^m \left\|B_jx^\ast_{{\bf h},N}\right\|_{H'}^{\lambda_j}} \geqslant \prod\limits_{j=0}^m \left(\frac{h_j}{\lambda_j}\right)^{2\lambda_j} \widetilde{\sum\limits_{n\in M_N}} \frac{\left|\left<f,Ae_{n}\right>\right|^2}{b_{n,\bf h}
		}.
	\]
	It remains to choose a sequence $\left\{{\bf h}^p\right\}_{p=1}^\infty\subset\mathbb{R}^{m+1}_+$ such that the function in the right hand part of the later inequality
	converges to its $\sup$ over ${\bf h}\in\mathbb{R}_+^{m+1}$.
\end{proof}

\subsubsection{Finiteness of constant $\mathcal{C}$}

Let us find some sufficient conditions that guarantee finiteness of constant $\mathcal{C}$ defined in~\eqref{condition_constant_new_abst}. Let $\overline{M} := \left\{n\in M\,:\,\forall j\in\{0,1,\ldots,m\}\Rightarrow\left\|B_je_n\right\|_{H'} \ne 0\right\}$. Observe that for finiteness of $\mathcal{C}$ it is necessary that $\left|\left<f, Ae_n\right>\right| = 0$ for every $n\in M\setminus \overline{M}$. 

We start with immediate corollary from AM-GM inequality. 

\begin{lemma}
\label{trivial_sufficient_condition}
	Let $\lambda_0,\ldots,\lambda_m > 0$, $\lambda_0 + \ldots + \lambda_m = 1$, be such that 
	\[
		\widetilde{\sum\limits_{n\in M}}\frac{\left|\left<f,Ae_{n}\right>\right|^2}{\left\|B_0e_n\right\|_{H'}^{2\lambda_0}\ldots\left\|B_me_n\right\|_{H'}^{2\lambda_m}} < \infty.
	\]
	Then constant $\mathcal{C}$ in~\eqref{condition_constant_new_abst} is finite. 
\end{lemma}

We proceed to less trivial sufficient condition. Let $d\in\mathbb{N}$ and $M\subset \mathbb{Z}^d$. For convenience, in what follows we denote elements of $\mathbb{R}^d$ in bold type ${\bf x} = \left(x_1,\ldots,x_d\right)$, where $x_1,\ldots,x_d$ are the coordinates of ${\bf x}\in \mathbb{R}^d$. 

Assume that there exist $C_1,C_2 > 0$, ${\bf k},{\bf r}^0,\ldots,{\bf r}^m\in \RR^d$ and sets $\left\{\alpha^j_{n}\right\}_{n \in \mathbb{Z}}$, $j=1,\ldots,d$, of non-negative numbers such that, for every ${\bf n}\in \overline{M}$, 
\begin{equation}
\label{multipliers_cond_01}
	\left|\left<f, Ae_{\bf n}\right>\right| \leqslant C_1\alpha_{\bf n}|{\bf n}|^{\bf k} := C_1\alpha^1_{n_1}\ldots\alpha^d_{n_d}|n_1|^{k_1}\ldots|n_d|^{k_d},
\end{equation}
\begin{equation}
\label{multipliers_cond_02}
	\left\|B_je_{\bf n}\right\|_{H'} \geqslant C_2|{\bf n}|^{{\bf r}^j} := C_2|n_1|^{r^j_1}\ldots|n_d|^{r^j_d}, \qquad j=0,1,\ldots,m.
\end{equation}
Denote by $\mathcal{S}\left({\bf r}^0,\ldots,{\bf r}^m\right)$ the convex hull of points ${\bf r}^0,\ldots,{\bf r}^d$ and by $\text{int}\,\Omega$ -- the interior of the set $\Omega\subset \RR^d$.

\begin{lemma}
\label{Lemma_constant_finiteness}
	Let $m\in\mathbb{Z}_+$, $d\in\mathbb{N}$, ${\bf k},{\bf r}^0,\ldots,{\bf r}^m\in\mathbb{R}^d$ be such that ${\bf k} + \frac 12\cdot{\bf 1}\in{\rm int}\,\mathcal{S}\left({\bf r}^0,\ldots,{\bf r}^m\right)$, and conditions~\eqref{multipliers_cond_01} and~\eqref{multipliers_cond_02} are fulfilled. Let also $\lambda_0,\ldots,\lambda_m > 0$ be such that $\lambda_0+\ldots+\lambda_m = 1$ and ${\bf k} + \frac{1}{2}\cdot {\bf 1} = \lambda_0 {\bf r}^0 + \ldots + \lambda_m{\bf r}^m$, and $\{\alpha_{\bf n}\}_{{\bf n}\in \overline{M}}$ be bounded. Then constant $\mathcal{C}$ in~\eqref{condition_constant_new_abst} is finite.
\end{lemma}

\begin{proof}
	For $j=1,\ldots,d$, let ${\bf 1}_j:=(0,\ldots,0,1,0,\ldots,0) \in\mathbb{R}^d$ be such that $1$ is located on the position with index $j$. There exists $\delta \in (0,1)$ such that points ${\bf s}^0 := {\bf k} + \frac{1-\delta}{2}\cdot {\bf 1}$ and ${\bf s}^j := {\bf k} + \frac 12\cdot{\bf 1} + \frac {\delta}2\cdot {\bf 1}_j$, $j=1,\ldots,d$, belong to ${\rm int}\,\mathcal{S}\left({\bf r}^0,\ldots,{\bf r}^m\right)$, and, for $j=0,\ldots,d$, there exist positive numbers $\lambda_{0}^j,\ldots,\lambda_m^j$ such that $\lambda_0^j + \ldots + \lambda_m^j = 1$, $\lambda_0^j{\bf r}^0 + \ldots + \lambda_m^j{\bf r}^m = {\bf s}^j$. Evidently, $\lambda_s^0 + \ldots + \lambda_s^d = (d+1)\lambda_s$, for every $s=0,\ldots,m$. Then, for every $j=0,\ldots,d$, ${\bf n}\in \overline{M}$ and ${\bf h}\in\mathbb{R}_+^{m+1}$, using the weighted AM-GM inequality we have
	\[
		b_{{\bf n}, \bf h} \geqslant C_2\sum\limits_{s=0}^m \lambda_{s}^j h_s |{\bf n}|^{2{\bf r}^s}\geqslant  C_2\prod\limits_{s=0}^m h_s^{\lambda_s^j} \cdot |{\bf n}|^{2\sum\limits_{s=0}^m \lambda_s^j{\bf r}^s} = C_2\prod\limits_{s=0}^m h_s^{\lambda_s^j} \cdot |{\bf n}|^{2{\bf s}^j}.
	\]
	Combining above inequalities over $j=0,\ldots,d$ and denoting $f_j := \prod\limits_{s=0}^m h_s^{\lambda_s^j}$, we obtain
	\[
		\prod\limits_{s=0}^m h_s^{\lambda_s} \widetilde{\sum\limits_{{\bf n}\in M}} \frac{\left|\left<f,Ae_{\bf n}\right>\right|^2}{b_{{\bf n},\bf h}
		} \leqslant \frac{C_1(d+1)}{C_2}\cdot \prod\limits_{j=0}^d f_j^{\frac 1{d+1}}\sum\limits_{{\bf n}\in \overline{M}} \frac{\alpha_{\bf n}^2 |{\bf n}|^{2{\bf k}}}{\sum\limits_{j=0}^d f_j |{\bf n}|^{2{\bf s}^j}}.
	\]
	Clearly, there exist $\xi,\tau_1,\ldots,\tau_d>0$ such that $f_0 = \xi\cdot \tau_1^{1-\delta}\ldots\tau_d^{1-\delta}$ and $f_j = \xi\cdot \tau_1\ldots\tau_d\cdot\tau_j^{\delta}$, $j=1,\ldots,d$. Without loss of generality we may assume that $|\alpha_{\bf n}|\leqslant 1$ for every ${\bf n}\in \overline{M}$. Then
	\[
		\prod\limits_{s=0}^m h_s^{\lambda_s} \widetilde{\sum\limits_{{\bf n}\in M}} \frac{\left|\left<f,Ae_{\bf n}\right>\right|^2}{b_{{\bf n},\bf h}
		} \leqslant \sum\limits_{{\bf n}\in \overline{M}}\frac{\frac{C_1(d+1)}{C_2}\cdot\tau_1\ldots\tau_d}{\prod\limits_{j=1}^d \left(\tau_j |{n_j}|\right)^{1-\delta} + \prod\limits_{j=1}^d \tau_j |{n_j}| \cdot \sum\limits_{j=1}^d \left(\tau_j|{n_j}|\right)^{(d+1)\delta}}.
	\]
	Taking $\sup$ over ${\bf h}\in\mathbb{R}_+^{m+1}$ in the left-hand part of above inequality and supremum over $\tau_1,\ldots,\tau_d > 0$ in its right-hand part, we obtain
	\[
		\mathcal{C} \!\leqslant\! \sup\limits_{\tau_1,\ldots,\tau_d > 0} \sum\limits_{{\bf n}\in \overline{M}}\frac{\frac{C_1(d+1)}{C_2}\cdot\tau_1\ldots\tau_d}{\prod\limits_{j=1}^d \left(\tau_j |{n_j}|\right)^{1-\delta} + \prod\limits_{j=1}^d \tau_j |{n_j}| \cdot \sum\limits_{j=1}^d \left(\tau_j|{n_j}|\right)^{(d+1)\delta}} =:\frac{C_1(d+1)}{C_2}\cdot\mathcal{D}.
	\]
	Next, for $e\subset \{1,\ldots,d\}$, let ${\bf q}^e = \left(q^e_1,\ldots,q^e_d\right)\in\mathbb{R}^d$ be such that $q^e_j = 1$ if $j\in e$ and $q^e_j = -1$ if $j\not\in e$. Obviously, there exists $\varepsilon > 0$ such that, for every $e\subset \{1,\ldots, d\}$, the point ${\bf k} + \frac{1}{2}\cdot {\bf q}^e$ belongs to ${\rm int}\,\mathcal{S}\left({\bf s}^0,\ldots,{\bf s}^d\right)$. 
	Then by the weighted AM-GM inequality,
	\[
		\prod\limits_{j=1}^d \left(\tau_j |{n_j}|\right)^{1-\delta} + \prod\limits_{j=1}^d \tau_j |{n_j}| \cdot \sum\limits_{j=1}^d \left(\tau_j|{n_j}|\right)^{(d+1)\delta} \geqslant \prod\limits_{j=1}^d \left(\tau_j |{n_j}|\right)^{q_j^e}. 
	\]
	Hence,
	\begin{gather*}
		\mathcal{D} \leqslant \sup\limits_{\tau_1,\ldots,\tau_d > 0} \sum\limits_{{\bf n}\in\mathbb{Z}^d\,:\,|{\bf n}|^{\bf 1}\ne 0} \frac{2^d\tau_1\ldots\tau_d}{\sum\limits_{e\subset\{1,\ldots,d\}} \prod\limits_{j=1}^d \left(\tau_j |{n_j}|\right)^{q_j^e}} \\ 
		= 2^{2d}\left(\sup\limits_{\tau > 0} \sum\limits_{n=1}^\infty\frac{\tau}{(\tau n)^{1-\varepsilon} + (\tau n)^{1+\varepsilon}}\right)^d.
	\end{gather*}
	As a result, we should prove only that the function
	\[
		f(\tau) := \sum\limits_{n=1}^\infty\frac{\tau}{(\tau n)^{1-\varepsilon} + (\tau n)^{1+\varepsilon}},\qquad \tau\in\mathbb{R}_+,
	\] 
	is bounded. Indeed, for $\tau\geqslant 1$,
	\[
		f(\tau) \leqslant \tau^{-\varepsilon}\cdot \sum\limits_{n=1}^\infty \frac{1}{n^{1+\varepsilon}} \leqslant \sum\limits_{n=1}^\infty \frac{1}{n^{1+\varepsilon}} <\infty,
	\]
	and, for $\tau \in (0,1)$, 
	\begin{gather*}
		f(\tau) \leqslant \sum\limits_{n\in \mathbb{N}\,:\,n < \frac 1{\tau}} \frac{\tau^\varepsilon}{n^{1-\varepsilon}} + \sum\limits_{j=1}^\infty \sum\limits_{n\in \mathbb{N}\,:\,\frac{j}{\tau}\leqslant n < \frac{j+1}{\tau}} \frac{\tau}{(\tau n)^{1+\varepsilon}} \\ 
		\leqslant \int_0^{\frac{1}{\tau}}\frac{\tau^{\varepsilon}{\rm d}t}{t^{1-\varepsilon}} + \sum\limits_{j=1}^\infty \frac{\tau}{j^{1+\varepsilon}}\left(\frac{1}{\tau} + 1\right) < \infty.
	\end{gather*}
	Combining above two inequalities we conclude that $\mathcal{D} < \infty$ and, hence, $\mathcal{C} < \infty$, which finishes the proof.
%
\end{proof}

\begin{rem}
\label{important_rem}
	Under assumptions of Lemma~\ref{Lemma_constant_finiteness}, by combining Lemma~\ref{trivial_sufficient_condition} with Lemma~\ref{Lemma_constant_finiteness}, we obtain that $\mathcal{C}<\infty$ for any $\lambda_0,\ldots,\lambda_m>0$ such that $\lambda_0+\ldots+\lambda_m=1$ and $\lambda_0{\bf r}^0 + \ldots + \lambda_m{\bf r}^m \leqslant {\bf k} + \frac{1}{2}\cdot{\bf 1}$.
\end{rem}

\begin{rem}
	Under specific choice of sequences of numbers $\left\{\alpha_n^j\right\}_{n\in \mathbb{Z}}$ the constant $\mathcal{C}$ in~\eqref{condition_constant_new_abst} can be finite for any $\lambda_0,\ldots,\lambda_m>0$ such that $\lambda_0+\ldots+\lambda_m = 1$. For example, let $m\in\mathbb{Z}_+$, $d\in\mathbb{N}$, ${\bf k},{\bf r}^0,\ldots,{\bf r}^m\in\mathbb{R}^d$ be such that ${\bf k} + \frac 12\cdot{\bf 1}\in{\rm int}\,\mathcal{S}\left({\bf r}^0,\ldots,{\bf r}^m\right)$, and conditions~\eqref{multipliers_cond_01} and~\eqref{multipliers_cond_02} are fulfilled. If there exist constants $\rho_1,\ldots\rho_d > 0$ such that $\alpha_n^j = e^{-\rho_j |n|}$, $n\in \mathbb{Z}$ and $j=1,\ldots,d$, then $\mathcal{C} < \infty$. 
\end{rem}

\subsection{Multiplicative Hardy-Littlewood-P\'olya inequality}

\begin{theorem}
	\label{HLP_for_special_multipliers_abst}
	Let $m\in\mathbb{Z}_+$, ${\bf h}\in \mathbb{R}_+^{m+1}$, operators $B_0,B_1,\ldots,B_m$ satisfy conditions~{\rm (B1)}--{\rm (B4)}, operator $A$ satisfy conditions~{\rm (A1)}--{\rm (A3)}. Also, let $\lambda_0,\lambda_1,\ldots,\lambda_m> 0$ be such that $\lambda_0+\lambda_1+\ldots+\lambda_m = 1$ and
	\begin{equation}
	\label{condition_constant_new_HLP_abst}
		\mathfrak{C}(A,{\bf B},\overline{\lambda}) = \mathfrak{C} := \sup\limits_{{\bf h}\in \mathbb{R}_+^{m+1}} \prod\limits_{j=0}^mh_j^{\lambda_j}\sup\limits_{f\in H':\atop \|f\|_{H'}\leqslant 1} \widetilde{\sum\limits_{n\in M}} \frac{\left|\left(f,Ae_n\right)_{H'}\right|^2}{b_{n,\bf h}
		} < \infty.
	\end{equation}
	Then, for every $x\in\mathcal{D}(A)\cap H_{\bf B}$, there holds sharp inequality
	\begin{equation}
	\label{HLP_final_multipliers_abst}
		\left\|Ax\right\|_{H'} \leqslant \sqrt{\mathfrak{C}\cdot\prod\limits_{j=0}^m\lambda_j^{-\lambda_j}} \cdot\prod\limits_{j=0}^m \left\|B_jx\right\|_{H'}^{\lambda_j}.
	\end{equation}
\end{theorem}

We can prove Theorem~\ref{HLP_for_special_multipliers_abst} using the arguments similar to those applied to prove Theorem~\ref{Taikov_inequality_for_multipliers_abst}. 

Let us present some sufficient conditions that guarantee finiteness of constant $\mathfrak{C}$ in~\eqref{condition_constant_new_HLP_abst}. To this end we assume additionally that:
\begin{equation}
\label{multipliers_conditions_1}
	\forall n',n''\in M, n'\ne n'',\quad \Rightarrow\quad \left(Ae_{n'}, Ae_{n''}\right)_{H'} = 0.
\end{equation}
From~\eqref{multipliers_conditions_1} it follows immediately that
\[
	\sup\limits_{f\in H':\atop \|f\|_{H'}\leqslant 1} \widetilde{\sum\limits_{n\in M}} \frac{\left|\left(f,Ae_n\right)_{H'}\right|^2}{b_{n,\bf h}
	} = \widetilde{\sup\limits_{n\in M}}\frac{\|Ae_n\|_{H'}^2}{b_{n,\bf h}
	}.
\]
Let $\overline{M} := \left\{n\in M\,:\,\forall j\in\{0,1,\ldots,m\}\Rightarrow\left\|B_je_n\right\|_{H'}\ne 0\right\}$. Evidently, for finiteness of constant $\mathfrak{C}$ it is necessary that $\left\|Ae_n\right\|_{H'} = 0$ for every $n\in M\setminus \overline{M}$.





Consider more specific sequences of $\left\|Ae_n\right\|_{H'}$'s and $\left\|B_je_n\right\|_{H'}$'s. Let $d\in\mathbb{N}$, $M\subset \ZZ^d$, ${\bf k},{\bf r}^0,\ldots,{\bf r}^m\in \mathbb{R}^d$. By $\mathcal{S}\left({\bf r}^0,\ldots,{\bf r}^m\right)$ we denote the convex hull of points ${\bf r}^0,{\bf r}^1,\ldots,{\bf r}^m$. Let also $\left\{g_n\right\}_{n\in \ZZ}\subset [0,+\infty)$ be any sequence of non-negative numbers and let
\begin{equation}
\label{multiplers_special_1}
	\left\|Ae_{\bf n}\right\|_{H'} = g_{\bf n}^{\bf k} := g^{k_1}_{n_1}\ldots g^{k_d}_{n_d}, \qquad {\bf n}\in M,
\end{equation}
\begin{equation}
\label{multipliers_special_2}
	\left\|B_je_{\bf n}\right\|_{H'} = g_{\bf n}^{{\bf r}^j} =  g_{n_1}^{r^j_1}\ldots g_{n_d}^{r^j_d}, \qquad j=0,1,\ldots,m,\quad{\bf n}\in M.
\end{equation}
Here we assume that $0^0 := 1$.

\begin{theorem}
\label{HLP_for_special_multipliers}
	Let $m\in\mathbb{Z}_+$, $d\in\mathbb{N}$, linear operators $B_0,B_1,\ldots,B_m$ satisfy conditions~{\rm (B1)},~{\rm (B2)},~{\rm (B4)}, linear operator $A$ satisfy conditions~{\rm (A1)},~{\rm (A3)} and~\eqref{multipliers_conditions_1}. Let also ${\bf k},{\bf r}^0,\ldots,{\bf r}^m\in\mathbb{R}^{d}$, $\lambda_0,\lambda_1,\ldots,\lambda_m> 0$ and sequence $\left\{g_n\right\}_{n\in M}\subset [0,+\infty)$ be such that $\lambda_0+\lambda_1+\ldots+\lambda_m = 1$, ${\bf k} = \lambda_0{\bf r}^0 + \ldots + \lambda_m{\bf r}^m\in\mathcal{S}\left({\bf r}^0,\ldots,{\bf r}^m\right)$, equalities~\eqref{multiplers_special_1} and~\eqref{multipliers_special_2} are satisfied, and there exist ${\bf n}_0\in M$ such that $g_{{\bf n}_0}^{\bf k} \ne 0$. Then, for every $x\in\mathcal{D}(A)\cap H_{\bf B}$, 
	\begin{equation}
	\label{HLP_final_multipliers}
		\left\|Ax\right\|_{H'} \leqslant \prod\limits_{j=0}^m \left\|B_jx\right\|_{H'}^{\lambda_j}.
	\end{equation}
	Inequality~\eqref{HLP_final_multipliers} turns into equality on every $e_{\bf n}$, ${\bf n}\in M$, such that $g_{\bf n}^{\bf k}\ne 0$.
\end{theorem}


\begin{proof}
	Since $g_{{\bf n}_0}^{\bf k}\ne 0$, by the weighted AM-GM inequality~\cite[pp.~74-75]{Cve_12} we conclude that operator $A$ satisfies condition~(A3) and operators $B_0,\ldots,B_m$ satisfy condition~(B3), as 
	\[
		g_{{\bf n}_0}^{2\bf k} = \prod\limits_{j=0}^m g_{{\bf n}_0}^{2\lambda_j{\bf r}^j} \leqslant \sum\limits_{j=0}^m \lambda_j g_{{\bf n}_0}^{2{\bf r}^j} \leqslant \sum\limits_{j=0}^m g_{{\bf n}_0}^{2{\bf r}^j}.
	\]
	Let us show that $\mathfrak{C} \leqslant \prod\limits_{j=0}^m \lambda_j^{\lambda_j}$. Indeed, for every ${\bf h}\in \RR_+^{m+1}$, applying the weighted AM-GM inequality we obtain
	\[
	    \begin{array}{rcl}
		    \displaystyle\widetilde{\sup\limits_{{\bf n}\in M}} \frac{\left\|Ae_{{\bf n}}\right\|_{H'}^2}{b_{{\bf n},\bf h}
		    } &= & \displaystyle \widetilde{\sup\limits_{{\bf n}\in M}} \frac{g_{\bf n}^{2{\bf k}}}{\sum\limits_{j=0}^m h_j g_{\bf n}^{2{\bf r}^j}} \leqslant \sup\limits_{{\bf x}\in \RR_+^{m+1}} \frac{{\bf x}^{2{\bf k}}}{\sum\limits_{j=0}^m h_j {\bf x}^{2{\bf r}^j}} = \displaystyle \sup\limits_{{\bf x}\in \RR_+^{m+1}}\frac{\prod\limits_{j=0}^m{\bf x}^{\lambda_j 2{\bf r}^j}}{\sum\limits_{j=0}^m h_j {\bf x}^{2{\bf r}^j}} \\
		    &= &\displaystyle \sup\limits_{{\bf x}\in \RR_+^{m+1}}\frac{\prod\limits_{j=0}^m \left(\frac{\lambda_j}{h_j}\right)^{\lambda_j}\prod\limits_{j=0}^{m} \left(\frac{h_j {\bf x}^{2{\bf r}^j}}{\lambda_j}\right)^{\lambda_j}}{\sum\limits_{j=0}^m h_j {\bf x}^{2{\bf r}^j}} \leqslant \displaystyle \prod\limits_{j=0}^m \left(\frac{\lambda_j}{h_j}\right)^{\lambda_j}.
		\end{array}
	\]
	Hence, 
	\[
		\mathfrak{C} = \sup\limits_{{\bf h}\in \mathbb{R}_+^{m+1}} \prod\limits_{j=0}^mh_j^{\lambda_j} \widetilde{\sup\limits_{{\bf n}\in M}} \frac{\left\|Ae_{\bf n}\right\|_{H'}^2}{b_{{\bf n},{\bf h}}} \leqslant\sup\limits_{{\bf h}\in \mathbb{R}_+^{m+1}} \prod\limits_{j=0}^m \left(\frac{\lambda_j}{h_j}\right)^{\lambda_j}\cdot \prod\limits_{j=0}^mh_j^{\lambda_j} \leqslant \prod\limits_{j=0}^m \lambda_j^{\lambda_j}.
	\]
	We conclude that conditions of Theorem~\ref{HLP_for_special_multipliers_abst} are satisfied and inequality~\eqref{HLP_final_multipliers} holds true. Straightforward calculations show that inequality~\eqref{HLP_final_multipliers} turns into equality on every $e_{\bf n}$, ${\bf n}\in M$, such that $g_{\bf n}^{\bf k}\ne 0$. 
\end{proof}

\begin{rem}
	Applying similar arguments as in the proof of Theorem~\ref{HLP_for_special_multipliers}, we can prove finiteness of $\mathfrak{C}$ in~\eqref{condition_constant_new_HLP_abst} in case there exist $C_1,C_2>0$, ${\bf k},{\bf r}^0,\ldots,{\bf r}^m\in\mathbb{R}^d$ and $\lambda_0,\ldots,\lambda_m>0$ such that $\lambda_0+\ldots+\lambda_m=1$, $\lambda_0{\bf r}^0 + \ldots + \lambda_m {\bf r}^m = {\bf k}\in \mathcal{S}\left({\bf r}^0,\ldots,{\bf r}^m\right)$ and, for ${\bf n}\in M$, 
	\[
		\|Ae_{\bf n}\|_{H'} \leqslant C_1g_{\bf n}^{\bf k}\quad\textrm{and}\quad\left\|B_je_{\bf n}\right\|_{H'}\geqslant C_2 g_{\bf n}^{{\bf r}^j},\quad j=0,1,\ldots,m.
	\]
\end{rem} 

\section{Applications}
\label{Sec:Applications}

This section is devoted to some applications of main results. First, we consider the problem of the best approximation of one class by elements from another class. Then we demonstrate how to obtain sharp multiplicative Hardy-Littlewood-P\'olya and Taikov types inequalities. More specifically, we will consider the case of the Laplace-Beltrami operator acting in $L_2(\mathcal{M})$, where $\mathcal{M}$ is $C^\infty$ Riemannian manifold without boundary, and the case of differential operators acting in $H = L_2(\mathbb{R}^d)$. 

Our considerations could be also applied to establish mean-squared versions of these inequalities. Moreover, using similar ideas we can obtain some other mean-squared and multiplicative inequalities, for example, inequalities established in~\cite{Raf_83,BerRaf_85,BabRas_00,BabKozSko_20}. 



\subsection{Approximation of a class by another class}

In~\cite{BabKofPic_97} (see also Theorem~7.4.1 in~\cite{BKKP}) there was established a general result on the equivalence of inequalities for the norms derivatives and some classical problems in Approximation Theory, in particular, the problem of the best approximation of one class by elements from another class. We apply this result to obtain direct consequences from Theorem~\ref{HLP_for_special_multipliers_abst}.

Let $H'=H$ and $B_0=\text{id}_H$. Consider the semi-norm $p(\cdot) = \|\cdot\|_{H}$ on $H$, and classes 
\[
    M_j := \left\{B_j^*z\,:\,\|z\|_{H}\leqslant 1, z\in\mathcal{D}\left(B_j^*\right)\right\},\qquad j=1,\ldots,m,
\]
\[
    M := \left\{A^*z\,:\,\|z\|_{H}\leqslant 1, z\in\mathcal{D}(A^*)\right\}\cap H_{\bf B}.
\]
Clearly, for every admissible $x\in H$, $S_{M_j}(x) := \sup\left\{(x,y)_{H}\,:\,y\in M_j\right\} = \left\|B_jx\right\|_{H}$ and $S_M(x) := \sup\left\{(x,y)_{H}\,:\,y\in M\right\} = \|Ax\|_{H}$.

Define the function $\Phi:\mathbb{R}_+^m\to\mathbb{R}$ as follows $\Phi({\bf t}) = K\prod\limits_{j=1}^m t_j^{\lambda_j}$, where $K > 0$, $\lambda_j > 0$ and $\lambda_0 := 1-\sum\limits_{j=1}^m \lambda_j > 0$, and set \[
    \overline{\Phi}({\bf t}) := \left\{
        \begin{array}{ll}
            -\Phi({\bf t}), & {\bf t}\in\mathbb{R}_+^m,\\
            +\infty, & {\bf t}\in\mathbb{R}^m\setminus\mathbb{R}_+^m.
        \end{array}
    \right.
\]
For the Legendre transformation of $\overline{\Phi}$ we have
\[
    \overline{\Phi}^*(-{\bf s}) := \sup\limits_{{\bf t}\in\mathbb{R}^m}\left\{-\sum\limits_{j=1}^m t_js_j - \overline{\Phi}({\bf t})\right\} = \left\{
    \begin{array}{ll}
        \lambda_0 K^{\frac{1}{\lambda_0}}\prod\limits_{j=1}^m \left(\frac{\lambda_j}{s_j}\right)^{\frac{\lambda_j}{\lambda_0}},& {\bf s}\!\in\!\mathbb{R}_+^m,\\
        +\infty, & {\bf  s} \!\in\!\mathbb{R}^m\!\setminus\!\mathbb{R}^m_+.
    \end{array}
    \right.
\]
Taking into account sharpness of inequality~\eqref{HLP_final_multipliers_abst} we obtain the following. 

\begin{cor}
    Under conditions of Theorem~\ref{HLP_for_special_multipliers_abst}, for every ${\bf t}\in\mathbb{R}_+^m$,
    \[
        E\left(M, \sum\limits_{j=1}^m t_jM_j\right)_H \!:=\! \sup\limits_{x\in M}\inf\limits_{u\in \sum\limits_{j=1}^mt_jM_j}\|x - u\|_H \!=\! \overline{\Phi}^*(-{\bf t}) \!=\! \lambda_0 K^{\frac{1}{\lambda_0}} \prod\limits_{j=1}^m \left(\frac{\lambda_j}{t_j}\right)^{\frac{\lambda_j}{\lambda_0}},
    \]
    where
    \[
        K = \sqrt{\mathfrak{C}\cdot\prod\limits_{j=0}^m\lambda_j^{-\lambda_j}}.
    \]
\end{cor}

\subsection{Inequalities for powers of the Laplace-Beltrami operators on compact Riemannian manifolds}

Following~\cite{Besse} we recall some basic notations and facts from Harmonic Analysis on compact Riemannian manifolds. Let $(\mathcal{M}, g)$ be a $d$-dimensional $C^\infty$ compact Riemannian manifold without boundary with volume element $\mu_g$. Given a local coordinate system $(y^i)$, let $(g_{ij})$ be the matrix with entries $g_{ij} = g\left(\frac{\partial}{\partial y^i}, \frac{\partial}{\partial y^j}\right)$, let $(g^{ij})$ denote the inverse matrix and put $\theta = \sqrt{\det{(g_{ij})}}$. Then the Laplace-Beltrami operator is locally given by
\[
	\Delta x = -\frac{1}{\theta} \sum\limits_{i,j=1}^d\frac{\partial}{\partial y^i}\left(\theta g^{ij}\frac{\partial x}{\partial y^{j}}\right),\qquad x\in C^\infty(\mathcal{M}).
\]
The operator $\Delta$ is formally self-adjoint and positive. It extends to a positive self-adjoint unbounded operator on
\[
	L_2(\mathcal{M}) = L_2(\mathcal{M},\mu_g) = \left\{x:\mathcal{M}\to\mathbb{C}\,:\,\|x\|_{L_2(\mathcal{M})}^2 \!:=\! \int_{\mathcal{M}}|x(t)|^2\,{\rm d}\mu_g(t) < \infty\right\}\!,
\]
and its extension, also denoted by $\Delta$, has a discrete spectrum $0=\mu_0^2<\mu_1^2\leqslant\mu_2^2\leqslant\ldots$ without accumulations, where each eigenvalue occurs as many times as its multiplicity. We denote by $\left\{\varphi_j\right\}_{j=0}^\infty$ an $L_2$-orthonormal basis of $C^\infty$-real eigenfunctions associated with the $\mu_j$'s. 

Given any $x\in L_2(\mathcal{M},\mu_g)$, we can write $x = \sum\limits_{j=0}^\infty a_j\varphi_j$ (with equality in $L^2$ sense), where $a_j = \left(x, \varphi_j\right)_{L_2(\mathcal{M})} = \int_{\mathcal{M}} x(t)\varphi_j(t)\,{\rm d}\mu_g(t)$. For $s\in\mathbb{R}$, we define the Sobolev space $H^s(\mathcal{M})$ as follows:
\[
	H^s(\mathcal{M}) := \left\{x:\mathcal{M}\to \mathbb{C}\,:\,\|x\|_{H^s(\mathcal{M})}^2 := \sum\limits_{j=1}^\infty |a_j|^2\mu_j^{2s} < \infty\right\},
\]
and set $H^s_0(\mathcal{M}) := \{x\in H^s(\mathcal{M})\,:\,a_0 = 0\}$.

Define fractional powers of $\Delta$. For $x\in H^{2s}_0(\mathcal{M})$, we set
\[
	\Delta^{s}x := \sum\limits_{j=1}^\infty a_j \mu_j^{2s} \varphi_j,
\] 
Clearly, $\Delta^s$ is self-adjoint positive and, hence, closed operator (see~\cite[\S7.3]{Yos_95}).

\subsubsection{Hardy-Littlewood-P\'olya type inequalities on compact Riemannian manifolds}

Since operator $\Delta^s : H_0^{2s}(\mathcal{M})\to H^0_0(\mathcal{M})$ satisfies conditions~(B1), (B2), (B4), (A1), (A3), we obtain the following corollary from Theorem~\ref{HLP_for_special_multipliers}.

\begin{theorem}
\label{HLP_Riemannian}
	Let $\mathcal{M}$ be a $d$-dimensional $C^\infty$ compact Riemannian manifold without boundary, $m\in\mathbb{Z}_+$, $r^0<\ldots< r^m\in \mathbb{R}$, $k\in \mathbb{R}$ and $\lambda_0,\ldots,\lambda_m > 0$ be such that $\lambda_0 + \ldots + \lambda_m = 1$ and $k = \lambda_0 r^0 + \ldots + \lambda_m r^m$. Then, for every $x\in H_0^{2r^m}(\mathcal{M})$, there holds true sharp inequality:
	\[
		\left\|\Delta^kx\right\|_{L_2(\mathcal{M})} \leqslant \prod\limits_{j=0}^m \left\|\Delta^{r^j}x\right\|_{L_2(\mathcal{M})}^{\lambda_j}.
	\]
	Any eigenfunction $\varphi_1,\varphi_2,\ldots$ turns above inequality into equality.
\end{theorem}


In case $\mathcal{M} = \mathbb{T}$ is period of length $2\pi$, we have that $\Delta = -D^2$. Denote by $D^s$, $s\in\mathbb{R}$, the Weyl fractional derivative of periodic function (see~\cite[\S19]{SamKilMar}). Evidently, 
\[
	\left\|D^sx\right\|_{L_2(\mathbb{T})} = \left\|\Delta^{\frac{s}{2}}x\right\|_{L_2(\mathbb{T})}.
\]
As a result, we obtain the following consequence from Theorem~\ref{HLP_Riemannian}.

\begin{cor}
\label{Cor_1}
	Let $m\in\mathbb{Z}_+$, $r^0<\ldots< r^m\in \mathbb{R}$, $k\in \mathbb{R}$ and $\lambda_0,\ldots,\lambda_m > 0$ be such that $\lambda_0 + \ldots + \lambda_m = 1$ and $k = \lambda_0 r^0 + \ldots + \lambda_m r^m$. Then, for every $x\in H^{r^m}_0(\mathbb{T})$, there holds true sharp inequality:
	\[
		\left\|D^kx\right\|_{L_2(\mathbb{T})} \leqslant \prod\limits_{j=0}^m \left\|D^{r^j}x\right\|_{L_2(\mathbb{T})}^{\lambda_j}.
	\]
	Any function $e^{int}$, $n\in\mathbb{Z}\setminus\{0\}$ and $t\in\mathbb{T}$, turns above inequality into equality.
\end{cor}

In case $m=1$, $r^0=0$ and $\lambda_1 = \frac{k}{r}$ statement of Corollary~\ref{Cor_1} is classical Hardy-Littlewood-P\'olya inequality~\cite{HarLitPol_34}.


Next, we consider multivariate analogues of Theorem~\ref{HLP_Riemannian}. Let $a\in\mathbb{N}$ and $\mathcal{M}^a = \mathcal{M}\times\ldots\times\mathcal{M}$ be tensor product of $a$ manifolds $\mathcal{M}$. For ${\bf s} = \left(s_1,\ldots,s_a\right)\in\mathbb{R}^a$, let $\mathbf{\Delta}^{\bf s} := \Delta_1^{s_1}\ldots\Delta_a^{s_a}$, where $\Delta_j$ is the Laplace-Beltrami operator acting on the $j$-th coordinate. Also, set $H_0^{\bf s}(\mathcal{M}^a) := H_0^{s_1}(\mathcal{M})\times\ldots\times H_0^{s_a}(\mathcal{M})$. From Theorem~\ref{HLP_for_special_multipliers} we can establish the following corollary.

\begin{theorem}
\label{HLP_Riemannian_mvt}
	Let $\mathcal{M}$ be a $d$-dimensional $C^\infty$ compact Riemannian manifold without boundary, $a\in \mathbb{N}$, $m\in\mathbb{Z}_+$, ${\bf k},{\bf r}^0,\ldots, {\bf r}^m\in \mathbb{R}^a$ be such that ${\bf k} \in\mathcal{S}\left({\bf r}^0, \ldots,{\bf r}^m\right)$. Also, let $\lambda_0,\ldots,\lambda_m > 0$ be such that $\lambda_0 + \ldots + \lambda_m = 1$ and ${\bf k} = \lambda_0 {\bf r}^0 + \ldots + \lambda_m {\bf r}^m$. Then, for every $x\in \bigcap\limits_{j=0}^m H^{2{\bf r}^j}_0(\mathcal{M}^a)$, 
	\[
		\left\|\mathbf{\Delta}^{\bf k}x\right\|_{L_2(\mathcal{M}^a)} \leqslant \prod\limits_{j=0}^m \left\|\mathbf{\Delta}^{{\bf r}^j}x\right\|_{L_2(\mathcal{M}^a)}^{\lambda_j}.
	\]
	Above inequality is sharp and turns into equality on each function $\varphi_{j_1}\times\ldots\times\varphi_{j_a}$ with $j_1\cdot\ldots\cdot j_a\ne 0$.
\end{theorem}


In case $\mathcal{M} = \mathbb{T}$, we denote ${\bf D}^{\bf s} := D^{s_1}_1\ldots D^{s_a}_a$. Then we can obtain the following consequence from Theorem~\ref{HLP_Riemannian_mvt}.

\begin{cor}
\label{Cor_2}
	Let $a\in \mathbb{N}$, $m\in\mathbb{Z}_+$, ${\bf k},{\bf r}^0,\ldots, {\bf r}^m\in \mathbb{R}^a$ be such that ${\bf k} \in\mathcal{S}\left({\bf r}^0, \ldots,{\bf r}^m\right)$. Also, let $\lambda_0,\ldots,\lambda_m > 0$ be such that $\lambda_0 + \ldots + \lambda_m = 1$ and ${\bf k} = \lambda_0 {\bf r}^0 + \ldots + \lambda_m {\bf r}^m$. Then, for every $x\in \bigcap\limits_{j=0}^m H^{{\bf r}^j}_0(\mathbb{T}^a)$, 
	\[
		\left\|\mathbf{D}^{\bf k}x\right\|_{L_2(\mathbb{T}^a)} \leqslant \prod\limits_{j=0}^m \left\|\mathbf{D}^{{\bf r}^j}x\right\|_{L_2(\mathbb{T}^a)}^{\lambda_j}.
	\]
	The above inequality is sharp and turns into equality on each function\break $e^{i(n_1t_1+\ldots+n_at_a)}$, where $n_1\cdot\ldots\cdot n_a\ne 0$ and $t_1,\ldots,t_d\in\mathbb{T}$.
\end{cor}


\subsubsection{Taikov type inequalities on compact Riemannian manifolds}

Let us establish consequences from Theorem~\ref{Taikov_for_multipliers_abst} for the powers of the Laplace-Beltrami operator. Let $\mathcal{M}$ be a $d$-dimensional $C^\infty$ compact Riemannian manifold without boundary. From a well-known theorem of H.~Weyl (see, {\it e.g.}~\cite[\S8]{Besse}) it follows that
\begin{equation}
\label{eq_0}
	\lim\limits_{j\to+\infty} j^{-\frac 1d}\mu_j = 2\pi \left(C(d){\rm Vol}\,(\mathcal{M})\right)^{-\frac 1d},
\end{equation}
where $C(d)$ is the volume of $d$-dimensional ball with unit radius, and\break ${\rm Vol}(\mathcal{M}) = \int_{\mathcal{M}}{\rm d}\mu_g(t)$. In turn, by L.~H\"ormander result (see~\cite{Hor_68}) there exists $C_3 > 0$ such that, for any eigenfunction $\phi$ corresponding to eigenvalue $\mu^2\ne 0$ of $\Delta$, 
\begin{equation}
\label{eq_1}
	\|\phi\|_{L_\infty(\mathcal{M})} := \textrm{essup}\left\{|\phi(t)|\,:\,t\in \mathcal{M}\right\} \leqslant C_3\mu^{\frac{d-1}{2}}\|\phi\|_{L_2(\mathcal{M})}.
\end{equation}

For $\xi\in\mathcal{M}$, consider linear functional $f_\xi\in\left(C(\mathcal{M})\right)^*$ mapping every $x\in C(\mathcal{M})$ into its value $x(\xi)$ at point $\xi$, {\it i.e.} $f_\xi x = x(\xi)$. Denote $\mathcal{C}_\xi := \mathcal{C}\left(f_\xi,\Delta^k,\Delta^{r^0},\ldots,\Delta^{r^m},\overline{\lambda}\right)$, where the constant on the right hand part was defined in~\eqref{condition_constant_new_abst}. Combining inequalities~\eqref{eq_1} and~\eqref{eq_0} we obtain
\[
	\mathcal{C}_\xi = \sup\limits_{{\bf h}\in \mathbb{R}^{m+1}_+} \prod\limits_{j=0}^mh_j^{\lambda_j} \sum\limits_{j=1}^\infty \frac{\mu_j^{4k}\varphi_{j}^2(\xi)}{\sum\limits_{l=0}^m h_l\mu_j^{4r^l}} \leqslant C_4\sup\limits_{{\bf h}\in \mathbb{R}^{m+1}_+} \prod\limits_{j=0}^mh_j^{\lambda_j} \sum\limits_{j=1}^\infty \frac{j^{\frac{4k+d-1}{d}}}{\sum\limits_{l=0}^m h_l j^{\frac{4r^l}{d}}}
\]
with some constant $C_4>0$ independent of $\xi$. Hence, there holds true the following corollary from Theorem~\ref{Taikov_for_multipliers_abst}, Lemma~\ref{Lemma_constant_finiteness} and Remark~\ref{important_rem}.

\begin{theorem}
\label{Taikov_Riemannian}
	Let $\mathcal{M}$ be a $d$-dimensional $C^\infty$ compact Riemannian manifold without boundary, $m\in\mathbb{N}$, $r^0<\ldots< r^m\in \mathbb{R}$, $k\in \mathbb{R}$ and $\lambda_0,\ldots,\lambda_m > 0$ be such that $\lambda_0 + \ldots + \lambda_m = 1$ and $k + \frac{2d-1}{4} \leqslant \lambda_0 r^0 + \ldots + \lambda_m r^m$. Then, for every $x\in H_0^{2r^m}(\mathcal{M})$, there holds sharp inequality:
	\[
		\left\|\Delta^{k}x\right\|_{C(\mathcal{M})} \leqslant \sqrt{\mathcal{C}\cdot\prod\limits_{j=0}^m\lambda_j^{-\lambda_j}}\cdot\prod\limits_{j=0}^m \left\|\Delta^{r^j}x\right\|_{L_2(\mathcal{M})}^{\lambda_j},
	\]	
	where
	\[
		\mathcal{C} = \sup\limits_{\xi\in\mathcal{M}}\sup\limits_{{\bf h}\in \mathbb{R}^{m+1}_+} \prod\limits_{j=0}^mh_j^{\lambda_j}\sum\limits_{j=1}^\infty \frac{\mu_j^{4k}\varphi_j^2(\xi)}{\sum\limits_{l=0}^m h_l\mu_j^{4r^l}}.
	\]
\end{theorem}

Consider an important example of compact Riemannian manifolds -- compact rank one symmetric spaces (CROSS). Let $\mathcal{M}$ be a CROSS, {\it i.e.} one of the following spaces: the Euclidean sphere $\mathbb{S}^b$, the projective space $\mathbb{R}P^b$, $\mathbb{C}P^b$, $\mathbb{H}P^b$, where $b\in\mathbb{N}$, or the Cayley plane $\mathbb{C}\text{a}P^2$. For convenience, let $0 = \gamma_0< \gamma_1 < \ldots $ be distinct eigenvalues of $\Delta^{\frac 12}$ and $\nu_0 = 1, \nu_1,\ldots$ be their multiplicities. Clearly, $\mu_0 = \gamma_0$, $\mu_1 = \ldots = \mu_{\nu_1} = \gamma_1$ and $\mu_{\nu_1+\ldots+\nu_r+1} = \ldots = \mu_{\nu_1+\ldots+\nu_r+\nu_{r+1}} = \gamma_r$ for $r\in\mathbb{N}$. Re-enumerate eigenfunctions $\varphi_0,\varphi_1,\ldots$ as $\left\{\phi_{j,l}\right\}_{j\in \ZZ_+,\,l=1,\ldots,\nu_j}$ to account for multiplicities of corresponding eigenvalues, {\it i.e.} $\Delta^{\frac 12}\phi_{j,l} = \gamma_j\phi_{j,l}$. It is well-known (see, {\it e.g.},~\cite{Kus_19}) that, for every $j\in\mathbb{N}$ and $\xi\in\mathcal{M}$, 
\[
	\sum\limits_{l=1}^{\nu_j} \phi_{j,l}^2(\xi) = \nu_j,
\]
where 
\[
    \nu_j = \frac{(2j+\alpha+\beta+1)\Gamma(\beta+1)\Gamma(j+\alpha+1)\Gamma(j+\alpha+\beta+1)}{\Gamma(\alpha+\beta+2)\Gamma(\alpha+1)\Gamma(j+1)\Gamma(j+\beta+1)},
\]
$\Gamma(t)$ is the Euler gamma function, and 
\begin{itemize}
	\item for $\mathcal{M} = \mathbb{S}^b$, $d=b$, $\gamma_j^2 = j(j+b-1)$ and $\alpha = \beta = \frac{d-2}{2}$;
	\item for $\mathcal{M} = \mathbb{R}P^b$, $d=b$, $\gamma_j^2 = 2j(2j+b-1)$ and $\alpha = \frac{d-2}{2}$, $\beta = -\frac{1}{2}$;
	\item for $\mathcal{M} = \mathbb{C}P^b$, $d = 2b$, $\gamma_j^2 = 4j(j+b)$ and $\alpha = \frac{d-2}{2}$, $\beta = 0$;
	\item for $\mathcal{M} = \mathbb{H}P^b$, $d=4b$, $\gamma_j^2 = 4j(j+2b+1)$ and $\alpha = \frac{d-2}{2}$, $\beta = 1$;
	\item for $\mathcal{M} = \mathbb{C}aP^2$, $d = 16$, $\gamma_j^2 = 4j(j+11)$ and $\alpha = \frac{d-2}{2}$, $\beta = 3$.
\end{itemize}

Clearly, there exists $C_5 > 0$ independent of $j$ and $\xi$ such that
\begin{equation}
\label{eq_2}
	\frac{1}{C_5}<\frac{1}{j^{d-1}}\sum\limits_{l=1}^{\nu_j}\phi^2_{j,l}(\xi) < C_5,
\end{equation}
and, hence, there exists $C_6 > 0$ independent of $\xi$ such that
\[
    \mathcal{C}_\xi = \sup\limits_{{\bf h}\in \mathbb{R}_+^{m+1}} \prod\limits_{j=0}^m h_j^{\lambda_j} \sum\limits_{j=1}^\infty \frac{\gamma_j^{4k}}{\sum\limits_{l=0}^m h_l\gamma_j^{4r^l}}\sum\limits_{l=1}^{\nu_j} \phi_{j,l}^2(\xi) \leqslant C_6 \sup\limits_{{\bf h}\in \mathbb{R}_+^{m+1}} \prod\limits_{j=0}^m h_j^{\lambda_j} \sum\limits_{j=1}^\infty \frac{j^{4k}}{\sum\limits_{l=0}^m h_l j^{4r^l}}.
\]

Combining Theorem~\ref{Taikov_for_multipliers_abst}, Lemma~\ref{Lemma_constant_finiteness} and Remark~\ref{important_rem}, we obtain.

\begin{theorem}
\label{Taikov_CROSS}
	Let $\mathcal{M}$ be a $d$-dimensional CROSS, $m\in\mathbb{N}$, $r^0<\ldots< r^m\in \mathbb{R}$, $k\in \mathbb{R}$ and $\lambda_0,\ldots,\lambda_m > 0$ be such that $\lambda_0 + \ldots + \lambda_m = 1$ and $k + \frac{d}{4} \leqslant \lambda_0 r^0 + \ldots + \lambda_m r^m$. Then, for every $x\in H_0^{2r^m}(\mathcal{M})$, there holds sharp inequality:
	\[
		\left\|\Delta^{k}x\right\|_{C(\mathcal{M})} \leqslant \sqrt{\mathcal{C}\cdot\prod\limits_{j=0}^m\lambda_j^{-\lambda_j}}\cdot\prod\limits_{j=0}^m \left\|\Delta^{r^j}x\right\|_{L_2(\mathcal{M})}^{\lambda_j},
	\]	
	where
	\[
	\mathcal{C} \!=\! \sup\limits_{{\bf h}\in \mathbb{R}^{m+1}_+} \prod\limits_{j=0}^mh_j^{\lambda_j}\sum\limits_{j=1}^\infty \frac{(2j\!+\!\alpha\!+\!\beta\!+\!1)\Gamma(\beta\!+\!1)\Gamma(j\!+\!\alpha\!+\!1)\Gamma(j\!+\!\alpha\!+\!\beta\!+\!1)\gamma_j^{4k}}{\Gamma(\alpha\!+\!\beta\!+\!2)\Gamma(\alpha\!+\!1)\Gamma(j\!+\!1)\Gamma(j\!+\!\beta\!+\!1)\left(\sum\limits_{l=0}^m h_l\gamma_j^{4r^l}\right)}.
	\]
\end{theorem}


In case $\mathcal{M} = \mathbb{S}^b$ the statement of Theorem~\ref{Taikov_CROSS} was proved by A.\,A.~Ilyin (see~\cite[Theorem~1.2]{Ily_98}). In case $\mathcal{M} = \mathbb{T}$ it is evident that $\left|\left<f_\xi, D^k e^{int}\right>\right| = \left|\left<f_\xi, \Delta^{\frac k2} e^{int}\right>\right|$. Hence, we obtain the following consequence from Theorem~\ref{Taikov_CROSS}.

\begin{cor}
\label{Cor_3}
	Let $m\in\mathbb{N}$, $r^0<\ldots< r^m\in \mathbb{R}$, $k\in \mathbb{R}$ and $\lambda_0,\ldots,\lambda_m > 0$ be such that $\lambda_0 + \ldots + \lambda_m = 1$ and $k + \frac{1}{2} \leqslant \lambda_0 r^0 + \ldots + \lambda_m r^m$. Then, for every $x\in H_0^{r^m}(\mathbb{T})$, there holds true sharp inequality:
	\[
		\left\|D^{k}x\right\|_{C(\mathcal{M})} \leqslant \sqrt{\mathcal{C}\cdot\prod\limits_{j=0}^m\lambda_j^{-\lambda_j}}\cdot\prod\limits_{j=0}^m \left\|D^{r^j}x\right\|_{L_2(\mathcal{M})}^{\lambda_j},
	\]	
	where
	\[
		\mathcal{C} = 2\sup\limits_{{\bf h}\in \mathbb{R}^{m+1}_+} \prod\limits_{j=0}^mh_j^{\lambda_j}\sum\limits_{j=1}^\infty \frac{j^{2k}}{\sum\limits_{l=0}^m h_l j^{2r^l}}.
	\]
\end{cor}

Corollary~\ref{Cor_3} generalizes the result of A.\,Yu~Shadrin~\cite{Sha_90}, who considered the case $m=1$, $r^0 = 0$ and $\lambda_1 = \frac{2k+1}{2r}$.


Next, we present multivariate analogues of above results. From Theorem~\ref{Taikov_for_multipliers_abst} we obtain the following corollary.

\begin{theorem}
	\label{Taikov_Riemannian_mvt}
	Let $\mathcal{M}$ be a $d$-dimensional $C^\infty$ compact Riemannian manifold without boundary, $a\in\mathbb{N}$, $m\in\mathbb{N}$, ${\bf k},{\bf r}^0,\ldots, {\bf r}^m\in \mathbb{R}^a$ be such that ${\bf k} + \frac{2d-1}{4} \cdot{\bf 1}\in{\rm int}\,\mathcal{S}\left({\bf r}^0,\ldots,{\bf r}^m\right)$. Also, let $\lambda_0,\ldots,\lambda_m > 0$ be such that $\lambda_0 + \ldots + \lambda_m = 1$ and ${\bf k} + \frac{2d-1}{4}\cdot{\bf 1} \leqslant \lambda_0 {\bf r}^0 + \ldots + \lambda_m {\bf r}^m$. Then, for every $x\in \bigcap\limits_{l=0}^m H_0^{2{\bf r}^l}(\mathcal{M}^a)$, there holds sharp inequality:
	\[
		\left\|{\bf \Delta}^{\bf k}x\right\|_{C(\mathcal{M}^a)} \leqslant \sqrt{\mathcal{C}\cdot\prod\limits_{j=0}^m\lambda_j^{-\lambda_j}}\cdot\prod\limits_{j=0}^m \left\|{\bf \Delta}^{{\bf r}^j}x\right\|_{L_2(\mathcal{M}^a)}^{\lambda_j},
	\]	
	where
	\[
		\mathcal{C} = \sup\limits_{\xi_1,\ldots,\xi_a\in\mathcal{M}^a}\sup\limits_{{\bf h}\in \mathbb{R}^{m+1}_+} \prod\limits_{j=0}^mh_j^{\lambda_j}\sum\limits_{{\bf j}\in\mathbb{N}^d} \frac{\mu_{j_1}^{4k_1}\ldots\mu_{j_a}^{4k_a}\varphi_{j_1}^2(\xi_1)\ldots\varphi_{j_a}^2(\xi_a)}{\sum\limits_{l=0}^m h_l\mu_{j_1}^{4r^l_1}\ldots\mu_{j_a}^{4r^l_a}}.
	\]
\end{theorem}

\begin{theorem}
\label{Taikov_CROSS_mvt}
	Let $\mathcal{M}$ be a $d$-dimensional CROSS, $a\in\mathbb{N}$, $m\in\mathbb{N}$ and ${\bf k},{\bf r}^0,\ldots, {\bf r}^m\in \mathbb{R}^a$ such that ${\bf k} + \frac{d}{4} \cdot{\bf 1}\!\in\!{\rm int}\,\mathcal{S}\left({\bf r}^0,\ldots,{\bf r}^m\right)$. Also, let $\lambda_0,\ldots,\lambda_m > 0$ be such that $\lambda_0 + \ldots + \lambda_m = 1$ and ${\bf k} + \frac{d}{4}\cdot{\bf 1} \leqslant \lambda_0 {\bf r}^0 + \ldots + \lambda_m {\bf r}^m$. Then, for every $x\in \bigcap\limits_{l=0}^m H^{2{\bf r}^l}_0(\mathcal{M}^a)$, there holds true sharp inequality:
	\[
		\left\|{\bf \Delta}^{\bf k}x\right\|_{C(\mathcal{M}^a)} \leqslant \sqrt{\mathcal{C}\cdot\prod\limits_{j=0}^m\lambda_j^{-\lambda_j}}\cdot\prod\limits_{j=0}^m \left\|{\bf \Delta}^{{\bf r}^j}x\right\|_{L_2(\mathcal{M}^a)}^{\lambda_j},
	\]	
	where
	\[
		\mathcal{C} = 2^a\sup\limits_{{\bf h}\in \mathbb{R}^{m+1}_+} \prod\limits_{j=0}^mh_j^{\lambda_j}\sum\limits_{{\bf j}\in\mathbb{N}^d} \frac{\gamma_{j_1}^{4k_1}\ldots\gamma_{j_a}^{4k_a}\rho_{j_1,\alpha,\beta}\ldots\rho_{j_a,\alpha,\beta}}{\sum\limits_{l=0}^m h_l\gamma_{j_1}^{4r^l_1}\ldots\gamma_{j_a}^{4r^l_a}},
	\]
	and
	\[
		\rho_{j,\alpha,\beta} := \frac{(2j+\alpha+\beta+1)\Gamma(\beta+1)\Gamma(j+\alpha+1)\Gamma(j+\alpha+\beta+1)}{\Gamma(\alpha+\beta+2)\Gamma(\alpha+1)\Gamma(j+1)\Gamma(j+\beta+1)}.
	\]
\end{theorem}


From Theorem~\ref{Taikov_CROSS_mvt} we obtain the following consequence for $\mathcal{M} = \mathbb{T}$.

\begin{cor}
\label{Cor_4}
	Let $a\in\mathbb{N}$, $m\in\mathbb{N}$, ${\bf k},{\bf r}^0,\ldots, {\bf r}^m\in \mathbb{R}^a$ be such that $k_1\cdot\ldots\cdot k_a\ne 0$ and ${\bf k} + \frac{1}{2} \cdot{\bf 1}\in{\rm int}\,\mathcal{S}\left({\bf r}^0,\ldots,{\bf r}^m\right)$. Also, let $\lambda_0,\ldots,\lambda_m > 0$ be such that $\lambda_0 + \ldots + \lambda_m = 1$ and ${\bf k} + \frac{1}{2}\cdot{\bf 1} \leqslant \lambda_0 {\bf r}^0 + \ldots + \lambda_m {\bf r}^m$. Then, for every $x\in \bigcap\limits_{l=0}^m H_0^{{\bf r}^l}(\mathbb{T}^a)$, there holds true sharp inequality:
	\[
		\left\|{\bf D}^{\bf k}x\right\|_{C(\mathbb{T}^a)} \leqslant \sqrt{\mathcal{C}\cdot\prod\limits_{j=0}^m\lambda_j^{-\lambda_j}}\cdot\prod\limits_{j=0}^m \left\|{\bf D}^{{\bf r}^j}x\right\|_{L_2(\mathbb{T}^a)}^{\lambda_j},
	\]	
	where
	\[
		\mathcal{C} = \sup\limits_{{\bf h}\in \mathbb{R}^{m+1}_+} \prod\limits_{j=0}^mh_j^{\lambda_j}\sum\limits_{{\bf j}\in\mathbb{N}^d} \frac{|{\bf j}|^{2{\bf k}}}{\sum\limits_{l=0}^m h_l|{\bf j}|^{2{\bf r}^l}}.
	\]
\end{cor}

Remark that Corollary~\ref{Cor_4} for two-dimensional torus, {\it i.e.} in the case $a=2$, was proved in~\cite{IlyTit_06}. In addition, boundedness of constant $\mathcal{C}$ also follows from Theorem~2b in~\cite{BusTih_79}. 


\subsection{Multiplicative inequalities for functions on $\mathbb{R}^d$}

We follow paper~\cite{BusTih_79} to introduce necessary notations. Let $d\in\mathbb{N}$, ${\bf k}\in \mathbb{R}^d$; $\mathcal{F}:L_2\left(\mathbb{R}^d\right)\to L_2\left(\mathbb{R}^d\right)$ be the Fourier transform mapping a function $x:\mathbb{R}^d\to\mathbb{C}$ into the function:
\[
	\widehat{x}({\bf u})=(\mathcal{F}x)({\bf u}) := \int_{\mathbb{R}^d}x({\bf t})e^{-2\pi i(t_1u_1+\ldots+t_du_d)}\,{\rm d}{\bf t},\qquad {\bf u}\in\mathbb{R}^d;
\] 
$\mathcal{F}^{-1}$ be the inverse Fourier transform mapping a function $x:\mathbb{R}^d\to\mathbb{C}$ into
\[
	\widetilde{x}({\bf t}) = \left(\mathcal{F}^{-1}x\right)({\bf t}) := \int_{\mathbb{R}^d}x({\bf u})e^{2\pi i(t_1u_1+\ldots+t_du_d)}\,{\rm d}{\bf u},\qquad{\bf t}\in\mathbb{R}^d;
\]
$\mathcal{E}^{\bf k}$ be the multiplication operator mapping a function $x:\mathbb{R}^d\to\mathbb{C}$ into the function $\left(\mathcal{E}^{\bf k}x\right)({\bf t}) =  (2\pi it_1)^{k_1}\!\ldots\!(2\pi it_d)^{k_d}x({\bf t})$, where $(2\pi ia)^b \!=\! |2\pi a|^be^{\frac{i\pi b\,{\rm sign}\,a}{2}}$; $\mathcal{D}^{\bf k} = \mathcal{F}^{-1}\circ\mathcal{E}^{\bf k}\circ \mathcal{F}$. Evidently, for ${\bf k}\in\mathbb{Z}^d_+$, $\mathcal{D}^{\bf k}$ is mixed partial differentiation operator of order ${\bf k}$. Denote by $\mathcal{L}^{\bf k}$ the Sobolev space of measurable functions $x:\mathbb{R}^d\to\mathbb{C}$ such that
\[
	\|x\|_{\mathcal{L}^{\bf k}}^2 := (2\pi)^{2k_1+\ldots+2k_d}\int_{\mathbb{R}^d}{\bf t}^{2{\bf k}}\left|\widehat{x}({\bf t})\right|^2\,{\rm d}{\bf t} < \infty.
\]
Obviously, $\mathcal{L}^{\bf 0} = L_2\left(\mathbb{R}^d\right)$. Operator $\mathcal{D}^{\bf k}$ is well defined on $\mathcal{L}^{\bf k}$ and by the Plancherel theorem $\left\|\mathcal{D}^{\bf k}x\right\|_{L_2(\mathbb{R}^d)} = \|x\|_{\mathcal{L}^{\bf k}}$ for every $x\in\mathcal{L}^{\bf k}$. 

Let $\ZZ_*^d = \left\{{\bf n}\in\ZZ^d\,:\,n_1\cdot\ldots\cdot n_d\ne 0\right\}$. For $\rho>0$, let $\left\{\mathfrak{c}_{{\bf n},\rho}\right\}_{{\bf n}\in\mathbb{Z}^d}$ be the set of cubes $\mathfrak{c}_{{\bf n},\rho} := [\rho n_1-\frac {\rho}2,\rho n_1+\frac{\rho}{2}]\times\ldots\times [\rho n_d - \frac{\rho}{2},\rho n_d+\frac{\rho}{2}]$, and by $L_{2;\rho}\left(\RR^d\right)$ we denote the space of compactly supported functions $x:\RR^d\to\mathbb{C}$ constant on cubes $\mathfrak{c}_{{\bf n},\rho}$, ${\bf n}\in \mathbb{Z}_*^d$, and vanishing on the cubes $\mathfrak{c}_{{\bf n},\rho}$ with ${\bf n}\in \mathbb{Z}^d\setminus\mathbb{Z}_*^d$. By $\mathcal{L}_{\rho}$ denote the space of the inverse Fourier transforms $\mathcal{F}^{-1}x$, $x\in L_{2;\rho}\left(\RR^d\right)$. It is clear that $\left\{\varphi_{{\bf n},\rho}\right\}_{{\bf n}\in\mathbb{Z}_*^d}$, where $\varphi_{{\bf n},\rho} = \rho^{-\frac{d}{2}}\cdot\mathcal{F}^{-1}\chi_{\mathfrak{c}_{{\bf n},\rho}}$ and $\chi_E$ is the indicator function of a measurable set $E\subset\mathbb{R}^d$, forms orthonormal basis in $\mathcal{L}_\rho$.

\subsubsection{Hardy-Littlewood-P\'olya type inequalities on $\mathbb{R}^d$}

Evidently, for ${\bf k}\in\mathbb{R}^d$,
\[
	\left\|\mathcal{D}^{\bf k}\varphi_{{\bf n},\rho}\right\|_{L_2(\mathbb{R}^d)} = (2\pi\rho)^{k_1+\ldots+k_d}\left(\int_{{\bf c}_{{\bf n},1}}|{\bf t}|^{2{\bf k}}\,{\rm d}{\bf t}\right)^{\frac 12}.
\]

It is not difficult to verify that operator $\mathcal{D}^{\bf k}:\mathcal{L}_\rho\to\mathcal{L}_\rho$ is self-adjoint and, hence, closed. Since $\mathcal{D}^{\bf k}$ satisfies conditions~(B1), (B2), (B4), (A1), (A3) and $\mathfrak{L}:=\bigcup\limits_{\rho > 0} \mathcal{L}_\rho$ is dense in $\mathcal{L}^{\bf k}$, we obtain the following corollary from Theorem~\ref{HLP_for_special_multipliers_abst}.

\begin{theorem}
\label{HLP_Cor}
	Let $d\in\mathbb{N}$, $m\in\mathbb{Z}_+$, $\rho>0$, ${\bf k},{\bf r}^0,\ldots, {\bf r}^m\!\in\! \mathbb{R}^d$ and $\lambda_0,\ldots,\!\lambda_m \!>\! 0$ be such that $\lambda_0 + \ldots + \lambda_m = 1$ and ${\bf k} = \lambda_0 {\bf r}^0 + \ldots + \lambda_m {\bf r}^m\in\mathcal{S}\left({\bf r}^0,\ldots,{\bf r}^m\right)$. Then, for every $x\in \bigcap\limits_{j=0}^m \mathcal{L}^{{\bf r}^j}$, there holds sharp inequality:
	\[
		\left\|\mathcal{D}^{\bf k}x\right\|_{L_2(\mathbb{R}^d)} \leqslant \prod\limits_{j=0}^m \left\|\mathcal{D}^{{\bf r}^j}x\right\|_{L_2(\mathbb{R}^d)}^{\lambda_j}.
	\]
\end{theorem}

\begin{proof}
    Indeed, by Theorem~\ref{HLP_for_special_multipliers_abst}, for every $x \in \mathcal{L}_\rho$, $\rho > 0$, 
    \[
        \left\|\mathcal{D}^{\bf k}x\right\|_{L_2(\mathbb{R}^d)} \leqslant \sqrt{\mathfrak{C}\cdot\prod\limits_{j=0}^m\lambda_j^{-\lambda_j}}\cdot \prod\limits_{j=0}^m \left\|\mathcal{D}^{{\bf r}^j}x\right\|_{L_2(\mathbb{R}^d)}^{\lambda_j},
    \]
    where
    \begin{gather*}
        \mathfrak{C} = \sup\limits_{{\bf h}\in\mathbb{R}^{m+1}_+}\prod\limits_{j=0}^{m}h_j^{\lambda_j}\sup\limits_{{\bf n}\in \mathbb{Z}_*^d}\frac{\left\|\mathcal{D}^{\bf k}\varphi_{{\bf n},\rho}\right\|^2_{L_2(\mathbb{R}^d)}}{\sum\limits_{j=0}^mh_j\left\|\mathcal{D}^{{\bf r}^j}\varphi_{{\bf n},\rho}\right\|^2_{L_2(\mathbb{R}^d)}} \\ 
        = \sup\limits_{{\bf h}\in\mathbb{R}^{m+1}_+}\prod\limits_{j=0}^{m}h_j^{\lambda_j}\sup\limits_{{\bf n}\in \mathbb{Z}_*^d}\frac{\left\|\mathcal{D}^{\bf k}\varphi_{{\bf n},1}\right\|^2_{L_2(\mathbb{R}^d)}}{\sum\limits_{j=0}^mh_j\left\|\mathcal{D}^{{\bf r}^j}\varphi_{{\bf n},1}\right\|^2_{L_2(\mathbb{R}^d)}}.
    \end{gather*}
    Let us show that $\mathfrak{C} = \prod\limits_{j=0}^m\lambda_j^{\lambda_j}$. By the weighted AM-GM inequality and the generalized H\"older inequality, 
    \begin{gather*}
	    \mathfrak{C}\leqslant \prod\limits_{j=0}^m\lambda_j^{\lambda_j} \cdot\sup\limits_{{\bf n}\in \mathbb{Z}^d_*}\frac{\left\|\mathcal{D}^{\bf k}\varphi_{{\bf n},1}\right\|^2_{L_2(\mathbb{R}^d)}}{\prod\limits_{j=0}^m\left\|\mathcal{D}^{{\bf r}^j}\varphi_{{\bf n},1}\right\|^{2\lambda_j}_{L_2(\mathbb{R}^d)}} \\ = \prod\limits_{j=0}^m\lambda_j^{\lambda_j}\cdot \sup\limits_{{\bf n}\in \mathbb{Z}_*^d}\frac{\int_{{\bf c}_{{\bf n},1}}|{\bf t}|^{2{\bf k}}\,{\rm d}{\bf t}}{\prod\limits_{j=0}^m\left(\int_{{\bf c}_{{\bf n},1}}|{\bf t}|^{2{\bf r}^j}\,{\rm d}{\bf t}\right)^{\lambda_j}}\leqslant \prod\limits_{j=0}^m\lambda_j^{\lambda_j}.
    \end{gather*}
    Since
    \[
	    \frac{\int_{{\bf c}_{{\bf n},1}}|{\bf t}|^{2{\bf k}}\,{\rm d}{\bf t}}{\prod\limits_{j=0}^m\left(\int_{{\bf c}_{{\bf n},1}}|{\bf t}|^{2{\bf r}^j}\,{\rm d}{\bf t}\right)^{\lambda_j}}\to 1,\qquad\text{as}\quad \min{\{n_1,\ldots,n_d\}}\to\infty,
    \]
    we conclude that $\mathfrak{C} = \prod\limits_{j=0}^m\lambda_j^{\lambda_j}$. Observing that $\mathfrak{L}$ is dense in $\mathcal{L}^{\bf k}$ and $\mathcal{L}^{{\bf r}^j}$, $j=0,1,\ldots,m$, we finish the proof.
\end{proof}



Remark that Theorem~\ref{HLP_Cor} contains the classical Hardy-Littlewood-P\'olya inequality~\cite{HarLitPol_34} (case $d=1$, $m=1$ and ${\bf r}^0 = {\bf 0}$). 

\subsubsection{Taikov type inequalities on $\mathbb{R}^d$}

Consider the functional $f_0:C(\mathbb{R}^d)\to\mathbb{C}$ mapping a function $x\in C(\mathbb{R}^d)$ into its value $x({\bf 0})$ at the point ${\bf 0}$. Observe that
\[
	\left|\left<f_{0},\mathcal{D}^{\bf k}\varphi_{{\bf n},\rho}\right>\right| = (2\pi)^{k_1+\ldots+k_d}\rho^{k_1+\ldots+k_d+\frac{d}{2}}\int_{\mathfrak{c}_{{\bf n},1}}|{\bf t}|^{\bf k}\,{\rm d}{\bf t}.
\]
Assume ${\bf k},{\bf r}^0,\ldots,{\bf r}^m\in\mathbb{R}^d$ be such that ${\bf k} + \frac 12\cdot{\bf 1} \in{\rm int}\,\mathcal{S}\left({\bf r}^0,\ldots,{\bf r}^m\right)$. Then the following embedding is well-known (see {\it e.g.}~\cite{BusTih_79}): $C^{\bf k}(\mathbb{R}^d)\subset \bigcap\limits_{j=0}^m \mathcal{L}^{{\bf r}^j}$. Also, it is not difficult to verify that $\mathfrak{L}$ is dense in $C^{\bf k}\left(\mathbb{R}^d\right)$. 
Hence, using intermediate approximation by space $\mathfrak{L}$, and applying Theorem~\ref{Taikov_for_multipliers_abst} and Lemma~\ref{Lemma_constant_finiteness} we obtain the following result.

\begin{theorem}
\label{Taikov_R_mvt}
	Let $d,m\in\mathbb{N}$, ${\bf k},{\bf r}^0,\ldots, {\bf r}^m\in \mathbb{R}^d$ and $\lambda_0,\ldots,\lambda_m > 0$ be such that $\lambda_0 + \ldots + \lambda_m = 1$ and ${\bf k} + \frac{1}{2}\cdot{\bf 1} = \lambda_0 {\bf r}^0 + \ldots + \lambda_m {\bf r}^m\in{\rm int}\,\mathcal{S}\left({\bf r}^0,\ldots,{\bf r}^m\right)$. Then, for every $x\in \bigcap\limits_{l=0}^m \mathcal{L}^{{\bf r}^l}$, there holds true sharp inequality:
	\[
		\left\|\mathcal{D}^{\bf k}x\right\|_{C(\mathbb{R}^d)} \leqslant \sqrt{\mathcal{C}\cdot\prod\limits_{j=0}^m\lambda_j^{-\lambda_j}}\cdot\prod\limits_{j=0}^m \left\|\mathcal{D}^{{\bf r}^j}x\right\|_{L_2(\mathbb{R}^d)}^{\lambda_j},
	\]	
	where
	\begin{gather*}
		\mathcal{C} = \liminf\limits_{\rho \to 0^+}\sup\limits_{{\bf h}\in \mathbb{R}^{m+1}_+} \prod\limits_{j=0}^mh_j^{\lambda_j}\sum\limits_{{\bf n}\in\mathbb{Z}_*^d} \frac{\rho^{d}\left(\int_{\mathfrak{c}_{{\bf n},1}} (2\pi \rho |{\bf t}|)^{\bf k}\,{\rm d}{\bf t}\right)^2}{\sum\limits_{l=0}^m h_l \int_{\mathfrak{c}_{{\bf n},1}} (2\pi \rho |{\bf t}|)^{2{\bf r}^l}\,{\rm d}{\bf t}} \\
		= \frac{1}{\pi^d}\sup\limits_{{\bf h}\in\mathbb{R}^{m+1}_+}\prod\limits_{j=0}^mh_j^{\lambda_j}\int_{\mathbb{R}^d_+}\frac{{\bf t}^{2{\bf k}}\,{\rm d\bf t}}{\sum\limits_{l=0}^mh_l{\bf t}^{2{\bf r}^l}}. 
	\end{gather*}
\end{theorem}

\begin{proof}
    It only remains to prove integral representation of constant $\mathcal{C}$. Indeed, let ${\bf h}\in\mathbb{R}^{m+1}_+$ be fixed. By the Cauchy-Schwarz inequality we have
    \begin{equation}
    \label{upper_estimate_C}
        \sum\limits_{{\bf n}\in\mathbb{N}^d} \frac{\left(\int_{\mathfrak{c}_{{\bf n},\rho}} |{\bf t}|^{\bf k}\,{\rm d\bf t}\right)^2}{\int_{\mathfrak{c}_{{\bf n},\rho}} \sum\limits_{l=0}^m h_l|{\bf t}|^{2{\bf r}^l}\,{\rm d\bf t}} \leqslant \sum\limits_{{\bf n}\in\mathbb{N}^d}\int_{\mathfrak{c}_{{\bf n},\rho}} \frac{{\bf t}^{2\bf k}\,{\rm d\bf t}}{\sum\limits_{l=0}^m h_l{\bf t}^{2{\bf r}^l}} \leqslant \int_{\mathbb{R}_+^d} \frac{{\bf t}^{2\bf k}\,{\rm d\bf t}}{\sum\limits_{l=0}^m h_l{\bf t}^{2{\bf r}^l}}.
    \end{equation}
    On the other hand, for every $\varepsilon > 0$ it is not difficult to show that there exists $\mu > 0$ and $N > 0$ satisfying properties: 
    \[
        \int_{\mathbb{R}^d_+\setminus [\mu,+\infty)^d}\frac{{\bf t}^{2{\bf k}}\,{\rm d\bf t}}{\sum\limits_{l=0}^mh_l{\bf t}^{2{\bf r}^l}} < \varepsilon,
    \]
    and, for every ${\bf n}\in\mathbb{N}^d$, such that $n_j > N$, $j=1,\ldots,d$, and every ${\bf y}\in\mathfrak{c}_{{\bf n},1}$,
    \[
        \int_{\mathfrak{c}_{{\bf n},1}}{\bf t}^{{\bf k}}\,{\rm d\bf t} \geqslant (1-\varepsilon) {\bf y}^{\bf k}.
    \]
    Now, let ${\bf y}_{{\bf n},\rho}\in\mathfrak{c}_{{\bf n},\rho}$ be the point at which the function $\sum\limits_{l=0}^m h_l\left|{\bf t}\right|^{2{\bf r}^l}$ attains its average value on the cube $\mathfrak{c}_{{\bf n},\rho}$. Then, for every sufficiently small $\rho > 0$,
    \begin{gather*}
    \label{lower_estimate_C}
        \sum\limits_{{\bf n}\in\mathbb{N}^d} \frac{\left(\int_{\mathfrak{c}_{{\bf n},\rho}} |{\bf t}|^{\bf k}\,{\rm d\bf t}\right)^2}{\int_{\mathfrak{c}_{{\bf n},\rho}} \sum\limits_{l=0}^m h_l|{\bf t}|^{2{\bf r}^l}\,{\rm d\bf t}} \geqslant (1-\varepsilon)^2\sum\limits_{{\bf n}\in\mathbb{N}^d:\atop n_j >\frac{\mu}{\rho} - \frac{1}{2}} \frac{\rho^d {\bf y}_{{\bf n},\rho}^{2{\bf k}}}{\sum\limits_{l=0}^m h_l{\bf y}_{{\bf n},\rho}^{2{\bf r}^l}} \\ \geqslant (1-\varepsilon)^3\int_{\mathbb{R}^d_+} \frac{{\bf t}^{2{\bf k}}\,{\rm d}{\bf t}}{\sum\limits_{l=0}^m h_l {\bf t}^{2{\bf r}^l}} - \varepsilon.
    \end{gather*}
    Hence, taking limit when $\rho\to 0^+$ and accounting for arbitrariness of $\varepsilon > 0$, we conclude from inequalities~\eqref{upper_estimate_C} and~\eqref{lower_estimate_C} that
    \[
        \mathcal{C} = \frac{1}{\pi^d}\sup\limits_{{\bf h}\in\mathbb{R}^{m+1}_+}\prod\limits_{j=0}^mh_j^{\lambda_j}\int_{\mathbb{R}^d_+}\frac{{\bf t}^{2{\bf k}}\,{\rm d\bf t}}{\sum\limits_{l=0}^mh_l{\bf t}^{2{\bf r}^l}},
    \]
    which finishes the proof.
\end{proof}

\begin{rem}
    In case $m=d$, for every ${\bf h}\in \mathbb{R}^{d+1}_+$ there exist ${\bf g} \in\mathbb{R}^{d}_+$ such that $\frac{h_j}{h_0} = {\bf g}^{2{\bf r}^m}$. Then altering variables ${\bf u} = {\bf g}{\bf t}$ in the integral representation of constant $\mathcal{C}$ in Theorem~\ref{Taikov_R_mvt}, we obtain
    \[
        \mathcal{C} = \int_{\mathbb{R}_+^d} \frac{{\bf u}^{2\bf k}\,{\rm d\bf u}}{\sum\limits_{l=0}^m{\bf u}^{2{\bf r}^l}}.
    \]
\end{rem}

L.\,V.~Taikov~\cite{Tai_68} first obtained Theorem~\ref{Taikov_R_mvt} in case $d=m=1$ and ${\bf r}^0 = {\bf 0}$ with explicit expression for constant $\mathcal{C}$. Theorem~\ref{Taikov_R_mvt} in case $m=d$ was established by A.\,P.~Buslaev and V.\,M.~Tihomirov (see~\cite[Theorem~3a]{BusTih_79}), and for every $m$ -- by A.\,A.~Ilyin~\cite{Ily_98}. Also, A.\,A.~Lunev~\cite{Lun_09} found explicit expressions for $\mathcal{C}$ for some specific values of ${\bf r}^0, \ldots, {\bf r}^m$. 


\section{Solyar-type inequalities}
\label{Sec:Solyar}

In this section we establish sharp Solyar-type inequality for closed unbounded operators acting into Hilbert space and having closed range. As a corollary we will prove sharp Solyar-type inequality for the powers of the Laplace-Beltrami operators on compact Riemannian manifolds and, in particular, obtain the result of paper~\cite{Sol_76}. More information and results on Solyar-type inequalities for the norms of derivatives can be found in~\cite{Sol_76,BabKofPich_96,BKKP,BabSam_03}. 

Let $X$ be normed space over $\mathbb{C}$, $H$ be the Hilbert space over $\mathbb{C}$ with scalar product $(\cdot,\cdot)$, $X^*$ be the dual space of $X$. Denote the value of functional $x^*\in X^*$ on the element $x\in X$ by $\left<x^*,x\right>$. Let $A:X\to H$ be the linear operator with dense domain $\mathcal{D}(A)\subset X$, and $A^*:H\to X^*$ be its dual. The following sharp inequality for bounded operators $A$ was proved in~\cite{BabSam_03}.

\begin{prop}
\label{Solyar_bounded}
    If $A$ is bounded then, for every $x\in X$,
    \[
        \|Ax\|_H^2 \leqslant \|x\|_X\|A^*Ax\|_{X^*}.
    \]
    This inequality is sharp in the following sense
    \[
        \sup\limits_{x\not\in \ker{A^*A}} \frac{\|Ax\|_H^2}{\|x\|_X\|A^*Ax\|_{X^*}} = 1.
    \]
\end{prop}

Although Proposition~\ref{Solyar_bounded} was proved originally for spaces $X$ and $H$ over field $\mathbb{R}$, its proof also remains true for spaces $X$ and $H$ over $\mathbb{C}$.

The following result generalizes Proposition~\ref{Solyar_bounded} for unbounded operators. 

\begin{theorem}
\label{Solyar_unbounded}
    Let $X$ be Banach space and $A$ be closed operator with closed range. Then, for every $x\in\mathcal{D}\left(A^*A\right)$, 
    \begin{equation}
    \label{Solyar_unbounded_inequality}
        \|Ax\|_H^2 \leqslant \|x\|_X\|A^*Ax\|_{X^*}.
    \end{equation}
    This inequality is sharp in the following sense
    \[
        \sup\limits_{x\not\in\ker{A}} \frac{\|Ax\|_H^2}{\|x\|_X\|A^*Ax\|_{X^*}} = 1.
    \]
\end{theorem}

\begin{proof}
    Indeed, for every $x\in\mathcal{D}\left(A^*A\right)$,
    \[
        \|Ax\|_H^2 = (Ax, Ax) = \left<A^*Ax, x\right> \leqslant \|x\|_X\|A^*Ax\|_{X^*},
    \]
    hence inequality~\eqref{Solyar_unbounded_inequality} follows. 
    
    Let us prove sharpness of inequality~\eqref{Solyar_unbounded_inequality}. For an operator $T$ denote by $\ker{T}$ its kernel and $R(T)$ -- its range, and by $E(z,M)_Z := \inf\limits_{u\in M}\|z-u\|_Z$ denote the best approximation of element $z$ in normed space $Z$ by a subset $M\subset Z$. Range $G = R(A)$ of operator $A$ is closed by assumption and its kernel $F = \ker{A}$ is closed as $A$ is closed. Also $R(A^*) = F^\perp$ (see, {\it e.g.}~\cite[\S5.5, Theorem]{Yos_95}). Consider quotient spaces $Y = X/F$ and $H' = H/G$ with natural norms $\|x+F\|_{Y} = E(x,F)_X$, $x\in X$, and $\|h+G\|_{H'} = E(h,G)_{H}$, $h\in H$. It is well-known that $Y$ is Banach space (see, {\it e.g.}~\cite[\S1.11]{Yos_95}). Consider operator $C:Y\to H'$ defined by the rule: $C(x+F) = Ax + G$, $x\in\mathcal{D}(A)$. Clearly, $\mathcal{D}(C) = \mathcal{D}(A) + F$, $C$ is closed operator, $\ker{C} = \{F\}$ consists of null element in $Y$ and $R(C) = H'$. Also, for every $h\in\mathcal{D}\left(A^*\right)$, $\left<C^*(h+G), x+F\right>_Y = \left<A^*h,x\right>$, $x\in X$ and $\left\|C^*(h+G)\right\|_{Y^*} = \left\|A^*h\right\|_{X^*}$.
    
    By Theorem in~\cite[\S5.5]{Yos_95}, $\ker{C^*} = \{G\}$ -- null element in $H'$ -- and $R\left(C^*\right) = Y^*$. Hence, by Corollary~1 in~\cite[\S5.5]{Yos_95} $C$ has bounded inverse $B:H'\to Y$. Then $B^*$ is also bounded (see~\cite[\S5.1, Theorem~2$'$]{Yos_95}). Let $\varepsilon\in(0,1)$. By Proposition~\ref{Solyar_bounded}, there exists $y^*\in Y^*$ such that
    \[
        \left\|B^*y^*\right\|_{H'}^2 > (1-\varepsilon) \left\|y^*\right\|_{Y^*}\left\|B^{**}B^*y^*\right\|_{Y^{**}}.
    \]
    Remark that $B^{**} = B$, as operator $B$ is bounded, and, since $BB^*y^*\in Y$, it follows that $\left\|BB^*y^*\right\|_{Y^{**}} = \left\|BB^*y^*\right\|_Y$. Since $R(C^*) = Y^*$ and $R(C) = H'$, there exist $h\in \mathcal{D}\left(A^*\right)$ and $x\in\mathcal{D}(A)$ such that $y^*=C^*(h+G)$ and $h+G=C(x+F)$. Evidently, $x\in\mathcal{D}\left(A^*A\right)$ and it is not difficult to verify that $B^*$ is inverse of $C^*$. Then $B^*x^*=B^*C^*(h+G)=h+G=C(x+F)$, $y^* = C^*C(x+F)$ and $BB^*y^*=BB^*C^*C(x+F)=BC(x+F)=x+F$. As a result, we obtain
    \[
        \left\|C(x+F)\right\|_{H'}^2 > (1-\varepsilon)\|x+F\|_{Y}\left\|C^*C(x+F)\right\|_{Y^*}.
    \]
    Next, let $x'\in x+F$ be such that $E(x,F)_X > (1-\varepsilon)\|x'\|_X$. Clearly, $x'\in\mathcal{D}\left(A^*A\right)$, $C(x+F)=Ax'+G$ and $\left\|C^*C(x+F)\right\|_{Y^*}=\left\|A^*Ax'\right\|_{X^*}$. This implies that
    \[
        \|Ax'\|_H^2 \geqslant \|Ax'+G\|_{H'}^2 = \|C(x+F)\|_{H'}^2 > (1-\varepsilon)^2\|x'\|_X\left\|A^*Ax'\right\|_{X^*},
    \]
    which finishes the proof.
\end{proof}


Let us now establish consequences from Theorem~\ref{Solyar_unbounded} for the powers of the Laplace-Beltrami operators on compact Riemannian manifolds $\mathcal{M}$ without boundary. For $1\leqslant p\leqslant \infty$, by $\overline{L}_p(\mathcal{M})$ denote the space of measurable functions $x:\mathcal{M}\to\mathbb{C}$ endowed with the norm
\[
    \|x\|_p = \|x\|_{\overline{L}_p(\mathcal{M})} := \left\{
        \begin{array}{ll}
            \displaystyle \left(\int_\mathcal{M}|x(t)|^p\,{\rm d}\mu_g(t)\right)^{\frac 1p}, & 1\leqslant p< \infty,\\
            \text{esssup}\,\left\{|x(t)|\,:\,t\in\mathcal{M}\right\}, & p=\infty.
        \end{array}
    \right.
\]
Following~\cite{Aubin}, we denote by $\overline{H}^k_p(\mathcal{M})$, $k\in\mathbb{N}$, the Sobolev space of functions $x:\mathcal{M}\to\mathbb{C}$ with the semi-norm
\[
    \|x\|_{\overline{H}_p^k(\mathcal{M})} := \left\|\Delta^{\frac{k}{2}}x\right\|_{p}.
\]


\begin{theorem}
\label{Solyar_type}
    Let $d\in\mathbb{N}$, $\mathcal{M}$ be a $d$-dimensional $C^\infty$ compact Riemannian manifold without boundary, $1\leqslant p < \infty$, $q = p/(p-1)$ and $k\in\mathbb{N}$ be such that $k\geqslant \frac{d}{2}\left(\frac{1}{2} - \frac{1}{p}\right)$. Then, for every $x\in \overline{H}^{4k}_{q}(\mathcal{M})$, there holds sharp inequality:
    \begin{equation}
    \label{Solyar_type_inequality}
        \left\|\Delta^kx\right\|_{2} \leqslant \sqrt{\|x\|_{p} \left\|\Delta^{2k}x\right\|_{q}}.
    \end{equation}
\end{theorem}

\begin{proof}
    Indeed, by~\cite[Theorem~2.20]{Aubin} and compactness of $\mathcal{M}$, chain of imbeddings $\overline{H}^{4k}_q(\mathcal{M})\subset \overline{H}^{2k}_2(\mathcal{M})\subset\overline{L}_p(\mathcal{M})$ follows. Clearly, operator $A = \Delta^k:\overline{L}_p(\mathcal{M})\to \overline{L}_2(\mathcal{M})$ with domain $\mathcal{D}\left(A\right) = \overline{H}^{2k}_2(\mathcal{M})$ is closed, attains all values from the closed subspace $\left\{x\in \overline{L}_2(\mathcal{M})\,:\,\int_{\mathcal{M}}x(t)\,{\rm d}\mu_g(t) = 0\right\}$, and is formally self-adjoint, {\it i.e.} $A^*x = Ax$ for every $x\in \overline{H}^{2k}_2(\mathcal{M})$. Hence, by Theorem~\ref{Solyar_unbounded}, for every $x\in\overline{H}^{4k}_q(\mathcal{M})$, sharp inequality~\eqref{Solyar_type_inequality} holds true.
\end{proof}

\begin{rem}
    Theorem~\ref{Solyar_type} contains the result by V.\,G.~Solyar~\cite{Sol_76} for $\mathcal{M} = \mathbb{T}$. Moreover, in case $k > \frac{d}{4}$, $\overline{H}^{2k}_2(\mathcal{M})\subset C(\mathcal{M})$ (see~\cite[Theorem~2.20]{Aubin}) and ideas from paper~\cite{Sol_76} can be used to prove that there exists extremal function in inequality~\eqref{Solyar_type_inequality} providing $p > 1$. 
\end{rem}

\begin{rem}
    Theorem~\ref{Solyar_type} holds true for fractional values of $k > \frac{d}{4}$, where $\overline{H}^{2k}_q(\mathcal{M})$ in case $2k\not\in\mathbb{N}$, stands for Bessel-potential spaces (see~\cite{Tri_85}).
\end{rem}

\section*{Acknowledgments}

This work was supported by Simons Collaboration Grant N.~210363.


\begin{thebibliography}{99}
    \bibitem{HarLit_12} G.\,H.~Hardy, J.\,E.~Littlewood, Contribution to the arithmetic theory of series, {\it Proc. London Math. Soc.}, 1912, {\bf s2-11}:1, 411--478.
	\bibitem{BKKP} V.\,F.~Babenko, N.\,P.~Korneichuk, V.\,A.~Kofanov, S.\,A.~Pichugov, Inequalities for derivatives and their applications, Kyiv: Naukova Dumka, 2003. 
	\bibitem{Lan_13} E.~Landau, Einige Ungleichungen fur zweimal differenzierbare Funktionen / E.~Landau, {\it Proc. London Math. Soc.}, 1913, {\bf 13}, 43--49.
	\bibitem{Had_14} J.~Hadamard, Sur le modul\'e maximum d'une fonction et de ses d\'eriv\'ees, {\it Soc. math. de France, Comptes rendus des S\'eances}, 1914.
	\bibitem{HarLitPol_34} G.\,H.~Hardy, J.\,E.~Littlewood, G.~P\'olya, Inequalities, Cambridge: University Press, 1934.
	\bibitem{Shi_37} G.\,E.~Shilov, On inequalities between derivatives, {\it Sb. scientific students works MGU}, 1937, 17--27.
	\bibitem{Kol_38} A.~Kolmogoroff, Une generalisation de l'in\'egalit\'e de M.\,J.~Hadamard entre les bornes superieures des d\'eriv\'ees successives d'une fonction, {\it C.\,R.~Acad. Sci.}, 1938, {\bf 207}, 764--765.
	\bibitem{Kol_39} A.\,N.~Kolmogorov, On inequalities between the upper bounds of the successive derivatives of an arbitrary function on the infinite interval, {\it Uch. Zap. MGU. Math.}, 1939, {\bf 30}:3, 3--16.
	\bibitem{Hor_55} L.~H\"ormander, On the theory of general partial differential operators, {\it Acta Math.}, 1955, {\bf 94}, 161--248.
	\bibitem{Yos_95} K.~Yosida, Functional Analysis, Berlin Heidelberg: Springer-Verlag, 1995, 504~p.
	\bibitem{Mit_Pec_Fin_91} D.\,S.~Mitrinovic, J.~Pecaric, A.\,M.~Fink, Inequalities Involving Functions and Their Integrals and Derivatives, Mathematics and its Applications, Springer Netherlands, {\bf 53}, 1991, 587~p.
	\bibitem{Kwo_Zet_92} M.\,K.~Kwong, A.~Zettl, Norm Inequalities for Derivatives and Differences, Series: Lecture Notes in Mathematics, Springer Berlin Heidelberg, 1992, 160~p.
	\bibitem{AreGab_95} V.\,V.~Arestov, V.\,N.~Gabushin, Best approximation of unbounded operators by bounded operators, {\it Russian Mathematics (Izvestiya VUZ. Matematika)}, 1995, {\bf 39}:11, 38--63.
	\bibitem{Are_96} V.\,V.~Arestov, Approximation of unbounded operators by bounded ones, and related extremal problems, {\it Russian Mathematical Surveys}, 1996, {\bf 51}:6, 1093--1126.
	\bibitem{BabKofPic_97} V.\,F.~Babenko, V.\,A.~Kofanov, S.\,A.~Pichugov, Multivariate inequalities of Kolmogorov type and their applications, {\it Proc. of Mannheim Conf. “Multivariate Approximation and Splines”, 1996 / G. Nurnberger, J. Schmidt, G.Walz (eds.).}, 1997, 1--12.
	\bibitem{BabKofPic_98} V.\,F.~Babenko, V.\,A.~Kofanov, S.\,A.~Pichugov, Inequalities of Kolmogorov Type and Some Their Applications in Approximation Theory, {\it Rendiconti del Circolo Matematico di Palermo Serie II, Suppl.}, 1998, {\bf 52}, 223--237.
	\bibitem{Ily_05} A.\,A.~Ilyin, Lieb-Thirring integral inequalities and their applications to attractors of the Navier-Stokes equations, {\it Sbornik: Mathematics}, 2005, {\bf 196}:1, 29--61.
	\bibitem{IlyTit_06} A.\,A.~Ilyin, E.\,S.~Titi, Sharp Estimates for the Number of Degrees of Freedom for the Damped-Driven 2-D Navier-Stokes Equations, {\it J. Nonlinear Science}, 2006, {\bf 16}:3, 233--253.
	\bibitem{Tai_68} L.\,V.~Taikov, Kolmogorov-type inequalities and the best formulas for numerical differentiation, {\it Math. Notes}, 1968, {\bf 4}:2, 631--634.
	\bibitem{BusTih_79} A.\,P.~Buslaev, V.\,M.~Tikhomirov, Inequalities for derivatives in the multidimensional case, {\it Math. Notes}, 1979, {\bf 25}:1, 32--40.
	\bibitem{Lub_60} Yu.\,I.~Lyubich, On inequalities between powers of a linear operator, {\it Izv. Akad. Nauk SSSR Ser. Mat.}, 1960, {\bf 24}:6, 825--864.
	\bibitem{Kup_75} N.\,P.~Kuptsov, Kolmogorov estimates for derivatives in $L_2[0,\infty)$, {\it Proc. Steklov Inst. Math.}, 1977, {\bf 138}, 101--125.
	\bibitem{Gab_69} V.\,N.~Gabushin, On the best approximation of the differentiation operator on the half-line, {\it Math. Notes}, 1969, {\bf 6}:5, 804--810.
	\bibitem{Kal_04} G.\,A.~Kalyabin, Sharp Constants in Inequalities for Intermediate Derivatives (the Gabushin Case), {\it Func. Anal. Appl.}, 2004, {\bf 38}:3, 184--191.
	\bibitem{Lun_Ori_09} A.\,A.~Lunev, L.\,L.~Oridoroga, Exact constants in generalized inequalities for intermediate derivatives, {\it Math. Notes}, 2009. {\bf 85}:5, 703--711.
	\bibitem{MagTih_81} G.\,G.~Magaril-Il'yaev, V.\,M.~Tihomirov, On the Kolmogorov inequality for fractional derivatives on the half-line, {\it Anal. Math.}, 1981, {\bf 7}:1, 37--47.
	\bibitem{Sha_90} A.\,Yu.~Shadrin, Inequalities of Kolmogorov type and estimates of spline interpolation on periodic classes $W^2_m$, {\it Math. Notes}, 1990, {\bf 48}:4, 1058--1063.
	\bibitem{Ily_98} A.\,A.~Ilyin, Best constants in a class of polymultiplicative inequalities for derivatives, {\it Sb. Math.}, 1998, {\bf 189}:9, 1335--1359.
	\bibitem{Lun_09} A.\,A.~Lunev, Exact constants in the inequalities for intermediate derivatives in $n$-dimensional space, {\it Math. Notes}, 2009, {\bf 85}:3, 458--462.
	\bibitem{BabBil_10} V.\,F.~Babenko, R.\,O.~Bilichenko, Taikov type inequalities for self-adjoint operators in Hilbert space, {\it Trudi IPMM NAN Ukraine}, 2010, {\bf 21}, 1--7.
	\bibitem{BabBil_11} V.\,F.~Babenko, R.\,O.~Bilichenko, Taikov type inequalities for powers of normal operators in Hilbert spaces, {\it Visnyk DNU, Ser. Matem.}, 2011, {\bf 16}, 3--7.
	\bibitem{BabKri_14} V.\,F.~Babenko, N.\,A.~Kriachko, On Hardy-Littlewood-P\'olya type inequalities for operators in Hilbert space, {\it Visnyk DNU, Ser. Matem.}, 2014, {\bf 19}, 1--5.
	\bibitem{BabBabKri_16} V.~Babenko, Yu.~Babenko, N.~Kriachko, Inequalities of Hardy-Littlewood-P\'olya Type for Functions of Operators and Their Applications, {\it J. Math. Anal. Appl.}, 2016, {\bf 444}:1, 512--526.
	\bibitem{Sol_76} V.\,G.~Solyar, An inequality for the norms of a function and of its derivatives, {\it Soviet Math. (Iz. VUZ)}, 1976, {\bf 20}:2, 58--62.
	\bibitem{BabLigShu_06} V.\,F.~Babenko, A.\,A.~Ligun, A.\,A.~Shumeiko, On the sharp inequalities of Kolmogorov type for operators in Hilbert spaces, {\it Vistnyk DNU, Ser. Matem.}, 2006, {\bf 11}, 9--13.
	\bibitem{Ste_67} S.\,B.~Stechkin, Best approximation of linear operators, {\it Math. Notes}, 1967, {\bf 1}:2, 91--99. 
	\bibitem{Bab_Bil_12} V.\,F.~Babenko, R.\,O.~Bilichenko, Approximation of unbounded functionals by bounded ones in Hilbert space, {\it Visnyk DNU, Ser. Matem.}, 2012, {\bf 17}, 3--10.
	\bibitem{DanSwa_63} N.~Dunford, J.\,T.~Schwartz, Linear operators. Part II: Spectral Theory, Interscience, New York, 1963.
	\bibitem{Cve_12} Z.~Cvetkovski, Inequalities: Theorems, Techniques and Selected Problems; Springer: Berlin/Heidelberg, Germany, 2012.
	\bibitem{Raf_83} S.\,Z.~Rafal'son, An inequality between the norms of a function and its derivatives in integral metrics, {\it Math. Notes}, 1983, {\bf 33}:1, 38--41.
	\bibitem{BerRaf_85} I.\,V.~Berdnikova, S.\,Z.~Rafal'son, Some inequalities between norms of a function and its derivatives in integral metrics, {\it Soviet Math. (Iz. VUZ)}, 1985, {\bf 29}:12, 1--5.
	\bibitem{BabRas_00} V.\,F.~Babenko, T.\,M.~Rassias, On Exact Inequalities of Hardy-Littlewood-Polya Type, {\it J. Math. Anal. Appl.}, 2000, {\bf 245}, 570--593.
	\bibitem{BabKozSko_20} V.\,F.~Babenko, O.\,V.~Kozynenko, D.\,S.~Skorokhodov, On Carlson type inequalities for spaces $L_{2,r;\alpha,\beta}((-1,1))$ and $L_{2,e^{-t^2}}(\mathbb{R})$, {\it Researches in Math.}, 2019, {\bf 27}:2, 45--58.
	\bibitem{Besse} A.\,L.~Besse, Manifolds all of whose Geodesics are Closed, A Series of Modern Surveys in Mathematics, Springer-Verlag, Berlin Heidelberg New York, 1978.
	\bibitem{SamKilMar} S.\,G.~Samko, A.\,A.~Kilbas, O.\,I.~Marichev, Fractional Integrals and Derivatives: Theory and Applications, CRC Press, 1993, 1006~p.
	\bibitem{Hor_68} L.~H\"ormander, The spectral function of an elliptic operator, {\it Acta Math.}, 1968, {\bf 121}, 193--218.
	\bibitem{Kus_19} A.\,K.~Kushpel, On the Lebesgue constants, {\it Ukr. Math. J.}, 2019, {\bf 71}:8, 1224--1233.
	\bibitem{BabKofPich_96} V.\,F.~Babenko, V.\,A.~Kofanov, S.\,A.~Pichugov, On inequalities of Landau-Hadamard-Kolmogorov type for $L_2$-norm of intermediate derivatives, {\it East J. Approx.}, 1996, {\bf 2}:3, 343--368.
	\bibitem{BabSam_03} V.~Babenko, O.~Samaan, On inequalities of Kolmogorov type for operators acting into Hilbert space and some applications, {\it Visnyk DNU, Ser. Matem.}, 2003, {\bf 8}, 11--18.
	\bibitem{Aubin} T.~Aubin, Some Nonlinear Problems in Riemannian Geometry, Springer-Verlag, 1998.
	\bibitem{Tri_85} H.~Triebel, Spaces of Besov-Hardy-Sobolev type on complete Riemannian manifolds, {\it Arkiv f\"or Matematik}, 1985, {\bf 24}:1-2, 299--337.
\end{thebibliography}
\end{document}